\documentclass[a4paper,12pt]{article}

\newenvironment{proof}{\noindent {\bf Proof:}}{\hfill $\Box$}

\newtheorem{theorem}{Theorem}
\newtheorem{lemma}{Lemma}

\newtheorem{assumption}{Assumption}

\usepackage{amsmath}
\usepackage{amssymb}
\usepackage{amsfonts}
\usepackage{graphicx}

\newcommand{\E}{\mathrm{\bf E}}

\title{\bf Convex computation of the maximum controlled invariant set for polynomial control systems$^\star$}

\begin{document}

\author{Milan Korda$^1$, Didier Henrion$^{2,3,4}$, Colin N. Jones$^1$}

\renewcommand\thefootnote{}
\footnotetext{\hspace{-0.2cm}$^\star$A preliminary version of this work, dealing with discrete-time systems only, has been submitted for possible presentation at the IEEE Conf. on Decision and Control, 2013.}
\footnotetext{\hspace{-0.2cm}$^1$Laboratoire d'Automatique, \'Ecole Polytechnique F\'ed\'erale de Lausanne, Station 9,
CH-1015, Lausanne, Switzerland. {\tt \{milan.korda,colin.jones\}@epfl.ch}}
\footnotetext{\hspace{-0.2cm}$^2$CNRS, LAAS, 7 avenue du colonel Roche, F-31400 Toulouse; France. {\tt henrion@laas.fr}}
\footnotetext{\hspace{-0.2cm}$^3$Universit\'e de Toulouse, LAAS, F-31400 Toulouse; France}
\footnotetext{\hspace{-0.2cm}$^4$Faculty of Electrical Engineering, Czech Technical University in Prague, Technick\'a 2, CZ-16626 Prague, Czech Republic}
\renewcommand\thefootnote{\arabic{footnote}}

\date{ \today}

\maketitle

\begin{abstract}
We characterize the maximum controlled invariant (MCI) set for discrete-
as well as continuous-time nonlinear dynamical systems as the solution of an infinite-dimensional
linear programming problem. For systems with polynomial dynamics
and compact semialgebraic state and control constraints, we describe a hierarchy
of finite-dimensional linear matrix inequality (LMI) relaxations whose optimal values converge to the volume of the MCI set; dual to these LMI relaxations are sum-of-squares (SOS) problems providing a converging sequence of outer approximations to the MCI set. The approach is simple and readily applicable in the sense that the approximations are the outcome of a single semidefinite program with no additional input apart from the problem description. A number of numerical examples illustrate the approach.
\end{abstract}

\section{Introduction}

Given a controlled dynamical system described by a differential (continuous-time)
or difference (discrete-time) equation, its maximum controlled invariant (MCI) set
is the set of all initial states that can be kept within a given constraint set ad~infinitum using admissible control inputs. This set goes by many other names in the literature, e.g., viability kernel in viability theory~\cite{aubin}, or $(A,B)$-invariant set in the linear case~\cite{dorea}.

Set invariance is an ubiquitous and essential concept in dynamical systems theory, as far as both analysis and control synthesis is concerned. In particular, by its very definition, the MCI set determines fundamental limitations of a given control system with respect to constraint satisfaction. In addition, there is a very tight link between invariant sets and (control) Lyapunov functions. Indeed, sub-level sets of a Lyapunov function give rise to invariant sets. Conversely, at least in the linear case, any controlled invariant set gives rise to a control Lyapunov function, and therefore these sets can be readily used to design stabilizing control laws; see, e.g., \cite{blanchini} for a general treatment and, e.g., \cite{gond_aut09, kerriganPHD} for applications in model predictive control design.

The problem of (maximum) controlled invariant set computation for discrete-time systems has been a topic of active research for more than four decades. The central tool in this effort has been the contractive algorithm of~\cite{bertsekas72} and its expansive counterpart~\cite{gutman_cwikel}. For an exhaustive survey and historical remarks see the survey~\cite{blanchini} and the book~\cite{miani}. 
 
Both algorithms, although conceptually applicable to any nonlinear system, have been predominantly applied in a linear setting where they boil down to a sequence of linear programs and polyhedral projections. Finite termination of this sequence is a subtle problem and sharp results are available only in the uncontrolled setting where no projections are required~\cite{gilbert_tan}; for discussion of finite-termination in the controlled case see~\cite{vidal}. The contractive and expansive algorithms were combined in~\cite{gond_aut09} to design an algorithm  terminating in a finite number of iterations and outputting an $\epsilon$-accurate inner approximation of the MCI set (with the accuracy measured by the Hausdorff distance). Another line of research, culminating in~\cite{rakovic}, exploits the linearity of the system dynamics in a more systematic way and approximates the maximum (or minimum) \emph{robust} controlled invariant set by the Minkowski sum of a parametrized family of sets. Very recently, in continuous time, \cite{kurzhanski_2013} developed a parallel algorithm for ellipsoidal approximations of the robust MCI set scalable to very high dimensions. Computation of low-complexity polyhedral controlled invariant sets was investigated in~\cite{blanchini2008} and \cite{blanco}.


In the nonlinear case, a common practice is to exploit the tight connection between invariance and Lyapunov functions and seek invariant sets as sub-level sets of a (control) Lyapunov function; see, e.g., \cite{chesi,topcu} and references therein for recent theoretical developments on the related problem of region of attraction computation and, e.g., \cite{anirudha} for practical applications of these techniques. This, however, typically leads to non-convex bilinear optimization problems which are notoriously hard to solve. Therefore, one often has to resort to ad-hoc analysis of the specific system at hand, which is typically tractable only in small dimensions; see~\cite{starkov1,starkov2} for concrete examples. Related in spirit is the localization technique of~\cite{kanatnikov} for discrete-time \emph{uncontrolled} systems, also requiring considerable effort in analysing the system.

Recently, a general approach using a hierarchy of finite-dimensional linear programs (LPs) was used in \cite{girard} to design a controller ensuring invariance of a given candidate polyhedral set. In our opinion, although being the current state of the art, this work still suffers from the following drawbacks: 
1) the sets obtained are convex polytopes (not general semi-algebraic sets, a fact particularly limiting in the nonlinear case where nonconvex MCI sets are common); 2) the geometry of the candidate polytopic
set must be given a priori; 3) there are no convergence guarantees to the MCI set.
In this paper, we explicitly address all these points.

Building upon our previous work \cite{roa} on the computation of
the region of attraction (ROA) for polynomial control systems, in this paper
we characterize the maximum controlled invariant (MCI) set for discrete-
as well as continuous-time polynomial systems as the solution to an infinite-dimensional LP problem in the cone of nonnegative measures. The dual of this problem is an infinite-dimensional LP in the space of continuous functions. Finite-dimensional \emph{relaxations} of the primal LP and finite-dimensional \emph{approximations} of the dual LP turn out to be semidefinite programs (SDPs) also related by duality. The primal relaxations lead to a truncated moment problem while the dual approximations to a sum-of-squares (SOS) problem. Super-level sets of one of the polynomials appearing in the dual SOS problem then provide outer approximations to the MCI set with guaranteed convergence as the degree of the polynomial tends to infinity.

The main mathematical tool we use are the so-called occupation measures which allow us to study the time evolution of the whole ensemble of initial conditions (described by a measure) rather than studying trajectories associated to each initial condition separately. The use of measures to study dynamical systems has a very long tradition: see~\cite{rubio} for probably the first systematic treatment\footnote{In~\cite{rubio}, J. E. Rubio used Young measures~\cite{young} rather than occupation measures, but the basic idea of ``linearizing'' a nonlinear problem by going into an infinite-dimensional space of measures is the same.}; for purely discrete-time treatment see~\cite[Chapter~6]{discrete}. To the best of the authors' knowledge our paper is the first one to use occupation measures for MCI set (approximate) \emph{computation}. The MCI set was previously \emph{characterized} using occupation measures in~\cite{quincampoix}, but there the characterization is rather indirect and not straightforwardly amenable to computation. Apart from the authors' work~\cite{roa}, the related problem of region of attraction computation was tackled using measures in~\cite{vaidya}. There, however, a very different approach was taken, not using \emph{occupation} measures but rather analyzing convergence via discretization of the state-space and propagating the initial distribution by means of a discretized transfer operator. Here, instead, we employ the (discounted) occupation measure which captures the behaviour of the trajectories emanating from the initial distribution over the infinite time horizon. As a result, our approach requires no discretization and, contrary to~\cite{vaidya}, provides true guarantees (not in an ``almost-everywhere'' or ``coarse'' sense) and, more importantly, is applicable in a controlled setting. Closely related to the occupation measures used here is the Rantzer's density~\cite{rantzer} which was used in~\cite{rajaram} to assess the stability of attractor sets of uncontrolled nonlinear systems. The approach, however, does not immediately yield approximations of the MCI set (or the region of attraction) and applies to uncontrolled systems only.


Similar in spirit to our approach, from the dual viewpoint of optimization over functions, are the Hamilton-Jacobi approaches (e.g., \cite{lygeros,mitchell}). However, contrary to these methods, our approach does not require state-space discretization and comes with convergence guarantees.

The contribution of our paper with respect to previous work on the topic can be summarized as follows:
\begin{itemize}
\item we deal with fully general continuous-time and discrete-time polynomial
dynamics under semi-algebraic state and control constraints;
\item our approximated MCI set is described by (the intersections of) polynomial super-level sets,
including more restrictive classes (e.g. polytopes, ellipsoids, etc.);
\item we provide a convex infinite-dimensional LP characterization of the MCI set;
\item we describe a hierarchy of convex finite-dimensional SDPs to solve the LP with convergence guarantees;
\item our approach is simple and readily applicable in the sense that the approximations are the result of a single SDP with \emph{no additional data} required apart from the problem description.
\end{itemize}
The contribution with respect to our previous work \cite{roa} can be summarised as follows:
\begin{itemize}
\item in \cite{roa} we compute the ROA, which is a related although different object: it is the set of all of initial conditions that can be steered to a given target set while satisfying state and control
constraints. In particular, the MCI set differs from the ROA in the sense that we do not try to hit any target set at a given time but rather try to keep the state within a given set forever. Therefore we had to adapt our technique to deal explicitly with invariance;
\item in \cite{roa} we dealt with continuous-time systems only, whereas
we can cope, with minor modifications, with discrete-time systems as well;
we choose to describe both the continuous-time and discrete-time setups
in parallel precisely to underline these common features;
\item in \cite{roa} we considered only a finite time-horizon, whereas
here we show how to cope, with the help of discounting, with an infinite horizon. This brought additional technical issues not encountered in finite time.
\end{itemize}

What can be considered a drawback of our approach is the fact that the approximations to the MCI set we obtain are from the outside and therefore not invariant. However, accurate outer approximations provide important information as to the performance limitations of the control system and are of practical interest, e.g., in collision avoidance. Therefore we believe that our work bears both theoretical and practical value, and naturally complements existing inner-approximation techniques.

The paper is organised as follows. The problem to be solved is described in Section~\ref{sec:probState}. Occupation measures are introduced in Section~\ref{sec:om}. The infinite-dimensional primal and dual LPs are described in Sections~\ref{sec:primalLP} and \ref{sec:dualLP}, respectively. The finite-dimensional relaxations with convergence results are presented in Section~\ref{sec:LMIrelax}. Numerical examples are in Section~\ref{sec:NumEx}. A reader interested only in the semialgebraic outer approximations of the MCI set can consult directly the infinite-dimensional dual LPs~(\ref{vlpd}) and (\ref{vlpc}) and their finite-dimensional approximations~(\ref{dlmid}) and (\ref{dlmic}) in discrete and continuous time, respectively.

\subsection{Notation}
Measures are understood as signed Borel measures on a Euclidean space, i.e., as countably additive maps from the Borel sets to the real numbers. From now on all subsets of a Euclidean space we refer to are automatically understood as Borel. The vector space of all signed Borel measures with its support
contained in a set $X$ is denoted by $M(X)$. The support (i.e., the smallest closed set
whose complement has a zero measure) of a measure~$\mu$ is denoted by $\mathrm{spt}\,\mu$.
The space of continuous functions on $X$ is denoted by $C(X)$ and likewise the space of
once continuously differentiable functions is $C^1(X)$. The indicator function of a set $X$
(i.e., a function equal to one on $X$ and zero otherwise) is denoted by $I_X(\cdot)$.
The symbol $\lambda$ denotes the $n$-dimensional Lebesgue measure (i.e.,  the standard
$n$-dimensional volume). The integral of a function $v$ with respect to a measure $\mu$ over a set $X$
is denoted by $\int_X v(x)\,d\mu(x)$. Sometimes for conciseness we use the shorter
notation $\int v\,d\mu$ omitting the integration variable and also the set over which
we integrate if they are obvious from the context. The ring of polynomials
in (possibly vector) variables $x_1$,\ldots,$x_n$ is denoted by $\mathbb{R}[x_1,\ldots,x_n]$.

\section{Problem statement}\label{sec:probState}
The approach is developed in parallel for discrete and continuous time.

\subsection{Discrete time}
Consider the discrete-time control system
\begin{equation}\label{sysd}
x_{t+1} = f(x_t,u_t), \quad x_t \in X, \quad u_t \in U, \quad t \in \{0,1,\dots\}
\end{equation}
with a given polynomial vector field $f$ with entries $f_i \in {\mathbb R}[x,u]$,
$i=1,\ldots,n$, and given compact basic semialgebraic state and input constraints
\[
\begin{array}{l}
x_t \in X := \{x \in {\mathbb R}^n \: :\: {g_X}_i(x) \geq 0, i=1,2,\ldots,n_X\}, \\ 
u_t \in U := \{u \in {\mathbb R}^m \: :\: {g_U}_i(u) \geq 0, i=1,2,\ldots,n_U\} \\
\end{array}
\]
with ${g_X}_i \in {\mathbb R}[x]$, ${g_U}_i \in {\mathbb R}[u]$.

The maximum controlled invariant (MCI) set is defined as
\begin{align*}
	X_I := \Big\{ x_0 \in X \: : \: & \exists\: \big(\{x_t\}_{t=1}^\infty, \{u_t\}_{t=1}^\infty\big)  \:\: \mathrm{s.t.}\:\: x_{t+1}=f(x_t,u_t),\\  \nonumber &u_t\in U,\,x_t\in X,\, \forall t\in \{0,1,\ldots\}\Big\}.
\end{align*}

A control sequence $\{u_t\}_{t=0}^\infty$ is called \emph{admissible} if $u_t\in U$ for all $t\in\{0,1,\dots\}$.

In words, the MCI set is the set of all initial states which can be kept inside the constraint set $X$ ad infinitum using admissible control inputs.

\subsection{Continuous time}

Consider the relaxed continuous-time control system
\begin{equation}\label{sysc}
\dot{x}(t) \in \mathrm{conv}\, f(x(t), U), \quad x(t) \in X, \quad t \in [0,\infty),
\end{equation}
where $\mathrm{conv}$ denotes the convex hull, $f$ is a polynomial vector field with entries $f_i \in {\mathbb R}[x,u]$,
$i=1,\ldots,n,$ and compact basic semialgebraic state and input constraint sets are defined~by
\[
\begin{array}{l}
X := \{x \in {\mathbb R}^n \: :\: {g_X}_i(x) \geq 0, i=1,2,\ldots,n_X\},\\ 
U := \{u \in {\mathbb R}^m \: :\: {g_U}_i(u) \geq 0, i=1,2,\ldots,n_U\} \\
\end{array}
\]
with ${g_X}_i \in {\mathbb R}[x]$, ${g_U}_i \in {\mathbb R}[u]$.
The meaning of the convex differential inclusion (\ref{sysc}) is as follows:
for all time $t$, the state velocity $\dot{x}(t)$ is constrained to the convex hull
of the set $f(x(t), U) := \{f(x(t),u) \: :\: u \in U\} \subset {\mathbb R}^n$.
The connection of this convexified (or relaxed) control problem~(\ref{sysc}) and the classical control problem $\dot{x}=f(x,u)$ is the Filippov-Wa$\dot{\mathrm z}$ewski Theorem \cite{aubin}, which shows that the trajectories of $\dot{x} = f(x,u)$ are dense (in the supremum norm) in the set of trajectories of the convexified inclusion\footnote{Note that the set $\mathrm{conv}\:f(x(t),U)$ is closed for every $t$ since $f$ is continuous and $U$ compact; therefore there is no need to take closure of the convex hull in order to apply the Filippov-Wa$\dot{\mathrm z}$ewski theorem.}~(\ref{sysc}). Therefore, from a practical point of view, there is little difference between the two formulations for the purposes of MCI set computation; see Section~3.2 and Appendices B and C of~\cite{roa} for a detailed discussion on this subtle issue. The simplest assumption under which the MCI sets for both systems coincide is $f(x,U)$ being convex for all $x$, which is in particular true for input-affine systems of the form $\dot{x} = f(x)+g(x)u$ with $U$ convex.

The maximum controlled invariant (MCI) set is defined as
\[
	X_I := \Big\{ x_0 \in X \: : \:  \exists\: x(\cdot)\:\: \mathrm{s.t.}\:\: \dot{x}(t)\in \mathrm{conv}\, f(x(t), U)\;\text{a.e.},\; x(t)\in X\, \forall\, t\in [0,\infty)\Big\},
\]
where $x(\cdot)$ is required to be absolutely continuous and \text{a.e.} stands for ``almost everywhere'' with respect to the Lebesgue measure on $[0,\infty)$.


In words, the MCI set is the set of all initial states for which there exists a trajectory of the convexified inclusion~(\ref{sysc}) which remains in $X$ ad infinitum.

\section{Occupation measures}\label{sec:om}

In this section we introduce the concept of occupation measures which is the centrepiece of our approach.

\subsection{Discrete time}
Given a discount factor $\alpha \in (0,1)$, an initial condition $x_0$ and an admissible control sequence $\{u_{t|x_0}\}_{t=0}^\infty$ such that the associated state sequence $\{x_{t|x_0}\}_{t=0}^\infty$ remains in $X$ for all time, we define the \emph{discounted occupation measure} $\mu(\cdot \mid x_0)\in M(X\times U)$ as
\begin{equation}\label{eq:discOM}
	\mu(A\times B \mid   x_0) := \sum_{t=0}^\infty \alpha^t I_{A\times B}(x_{t|x_0},u_{t|x_0})
 \end{equation}
 for all sets $A\subset X$ and $B\subset  U$.
 
 
In words, the discounted occupation measure measures the (discounted) number of visits of the state-control pair trajectory $(x(\cdot\!\mid\! x_0),\nu(\cdot\!\mid\! x_0))$ to subsets of $X\times U$. The discounting in the definition of the occupation measure ensures that $\mu(A\times B\mid x_0)$ is always finite; in fact we have $\mu(X\times U\mid x_0) = (1-\alpha)^{-1} $.
 
Now suppose that the initial condition is not a single point but an \emph{initial measure}\footnote{The initial measure $\mu_0$ can be thought of as the probability distribution of the initial state, although we do not require the mass of $\mu_0$ to be normalized to one.}  $\mu_0\in M(X)$ and an admissible control sequence is associated to each initial condition from the support of $\mu_0$ in such a way that the corresponding state sequence remains in $X$. Then we define the \emph{average discounted occupation measure} $\mu\in M(X\times U)$ as
 \[
  \mu(A\times B) := \int_X \mu(A\times B \!\mid \!   x_0)\,d\mu_0(x_0).
 \]
 
 The average discounted occupation measure measures the discounted average number of visits in subsets of $X\times U$ of trajectories starting from the initial distribution $\mu_0$.
 
 Now we derive an equation linking the measures $\mu_0$ and $\mu$. This equation will play a key role in subsequent development and in a sense replaces the dynamics equation~(\ref{sysd}). To derive this equation fix an initial condition $x_0\in X$ and a control sequence $\{u_{t|x_0}\}_{t=0}^\infty$ such that the associated state sequence $\{x_{t|x_0}\}_{t=0}^\infty$ stays in $X$. Then for any $v\in C(X)$ we have
\begin{align*}
\int_{X\times U} v(x)\,d\mu(x,u \!\mid\! x_0) &= \sum_{t=0}^\infty \alpha^t v(x_{t|x_0}) = v(x_{0|x_0}) + \alpha\sum_{t=0}^\infty \alpha^tv(x_{t+1|x_0})\\  
&= v(x_{0|x_0}) + \alpha\sum_{t=0}^\infty \alpha^t v(f(x_{t|x_0},u_{t|x_0}))\\ &= v(x_{0|x_0})  + \alpha \int_{X\times U}\hspace{-0.8em}v(f(x,u))\,d\mu(x,u \!\mid\! x_0).
\end{align*}
Integrating w.r.t. $\mu_0$ we arrive at the sought equation
\begin{equation}\label{eq:LiouvDisc}
\int_{X\times U} v(x)\, d\mu(x,u) =
\int_X v(x)\,d\mu_0(x) + \alpha\int_{X\times U} v(f(x,u))\,d\mu(x,u)\quad \forall v\in C(X).
\end{equation}

Note that this is an infinite-dimensional linear equation in variables $(\mu_0,\mu)$.

The following crucial Lemma establishes the connection between the support of any initial measure $\mu_0$ solving~(\ref{eq:LiouvDisc}) and the MCI set $X_I$.
\begin{lemma}\label{lem:corrDisc}
For any pair of measures $(\mu_0,\mu)$ satisfying equation~(\ref{eq:LiouvDisc}) with $\mathrm{spt}\,\mu_0\subset X$ and $\mathrm{spt}\,\mu\subset U\times X$ we have $\mathrm{spt}\,\mu_0 \subset X_I$.
\end{lemma}
\begin{proof}
A detailed proof is in Appendix~A.
\end{proof}

\subsection{Continuous time}
Given an initial condition $x_0$ and a trajectory $x(\cdot\!\mid\! x_0)$ of the inclusion~(\ref{sysc}) that remains in $X$ for all $t\ge0$, there exists an admissible time-varying measure-valued relaxed control $\nu_t(\cdot\!\mid\! x_0)\in M(U)$, $\nu_t(U\!\mid\!x_0) = 1$, such that
\[
\dot{x}(t) = \int_U f(x(t),u)\,d\nu_t(u\!\mid\! x_0)
\]
almost everywhere with respect to the Lebesgue measure on $[0,\infty)$. This follows from the definition of the convex hull (in fact, for each $t$, $\nu_t(\cdot\!\mid\! x_0)$ can be taken to be a convex combination of finitely many Dirac measures).

Then, given a discount factor $\beta > 0$, we define the \emph{discounted occupation measure} $\mu(\cdot\nobreak\mid\nobreak x_0)\in M(X\times U)$ as
\[
	\mu(A\times B \mid x_0) := \int_0^\infty \int_U e^{-\beta t} I_{A\times B}(x(t\!\mid\! x_0),u)\,d\nu_t(u\!\mid\! x_0)\,dt
 \]
 for all sets $A\subset X$ and $B\subset  U$.
 
In words, the discounted occupation measure measures the (discounted) time spent by the state-control pair trajectory $(x(\cdot\!\mid\! x_0),\nu(\cdot\!\mid\! x_0))$ in subsets of $X\times U$. The discounting in the definition of the occupation measure ensures that $\mu(A\times B\mid x_0)$ is always finite; in fact we have $\mu(X\times U\mid x_0) = \beta^{-1}$.
 
Now suppose that the initial condition is not a single point but an \emph{initial measure}\footnote{The initial measure $\mu_0$ can be thought of as the probability distribution of the initial state, although we do not require the mass of $\mu_0$ to be normalized to one.}  $\mu_0\in M(X)$ and a state trajectory that remains in $X$ along with an admissible relaxed control is associated to each initial condition from the support of $\mu_0$. Then we define the \emph{average discounted occupation measure} $\mu\in M(X\times U)$ as
 \[
  \mu(A\times B) := \int_X \mu(A\times B \mid x_0)\,d\mu_0(x_0).
 \]
 
Now we derive an equation linking the measures $\mu_0$ and $\mu$. This equation will play a key role in subsequent development and in a sense replaces the dynamics equation~(\ref{sysc}). To derive the equation, fix an initial condition $x_0\in X$, a trajectory $x(\cdot \!\mid\! x_0)$ that remains in $X$ with an associated admissible relaxed control $\nu_t(\cdot \!\mid\! x_0)$. Then for any $v\in C^1(X)$ integration by parts yields
\begin{align*}
\int_{X\times U}\mathrm{grad}\,v\cdot f(x,u)\,d\mu(x,u\!\mid\! x_0) &= \int_0^\infty\int_U e^{-\beta t}\mathrm{grad}\,v\!\cdot\! f(x(t\mid x_0),u)\,d\nu_t(u \!\mid\! x_0)\,dt\\  &= \int_0^\infty e^{-\beta t}\frac{d}{dt}v(x(t \!\mid\! x_0))\,dt\\ &= \beta \int_0^\infty e^{-\beta t}v(x(t \!\mid\! x_0))\,dt - v(x(0 \!\mid\! x_0))\\
&=  \beta \int_{X\times U} v(x)\,d\mu(x,u\!\mid\! x_0) - v(x(0\!\mid\! x_0)),
\end{align*}
where the boundary term at infinity vanishes due to discounting and the fact that $X$ is bounded. Integrating with respect to $\mu_0$ then gives the sought equation
\begin{equation}\label{eq:LiouvCont}
\beta\int_{X\times U} v(x)\, d\mu(x,u) =
\int_X v(x)\,d\mu_0(x) + \int_{X\times U} \mathrm{grad}\, v\cdot f(x,u)\,d\mu(x,u)\quad \forall v\in C^1(X).
\end{equation}

Note that this is an infinite-dimensional linear equation in variables $(\mu_0,\mu)$.

The following crucial Lemma establishes the connection between the support of any initial measure satisfying~(\ref{eq:LiouvCont}) and the MCI set $X_I$.
\begin{lemma}\label{lem:corrCont}
For any pair of measures $(\mu_0,\mu)$ satisfying equation~(\ref{eq:LiouvCont}) with $\mathrm{spt}\,\mu_0\subset X$ and $\mathrm{spt}\,\mu\subset U\times X$ we have $\lambda(\mathrm{spt}\,\mu_0) \le \lambda(X_I)$.
\end{lemma}
\begin{proof}
A detailed proof is in Appendix~B.
\end{proof}

\section{Primal LP}\label{sec:primalLP}
In this section we show how the MCI set computation problem can be cast as an infinite-dimensional
LP problem in the cone of nonnegative measures. As in~\cite{roa}, the basic idea is to maximize the mass of the initial measure $\mu_0$ subject to the constraint that it be dominated by the Lebesgue measure, that is, $\mu_0 \le \lambda$. System dynamics is captured by the equations~(\ref{eq:LiouvDisc}) and~(\ref{eq:LiouvCont}) for discrete and continuous times, respectively; state and input constraints are expressed through constraints on the supports of the initial and occupation measure. The constraint that  $\mu_0 \le \lambda$ can be equivalently rewritten as $\mu_0+\hat{\mu}_0=\lambda$ for some nonnegative slack measure $\hat{\mu}_0\in M(X)$. This constraint is in turn equivalent to $\int_X w(x)\,d\mu_0(x)+\int_X w(x)\,d\hat{\mu}_0(x) = \int_X w(x)\,d\lambda(x)$ for all $w\in C(X)$. These considerations lead to the following primal LPs.

\subsection{Discrete time}

The primal LP in discrete time reads
\begin{equation}\label{rrlpd}
\begin{array}{rclll}
p^* & = & \sup & \mu_0(X) \\
&& \mathrm{s.t.} & \int v(x)\,d\mu(x,u) = \int v(x)\,d\mu_0(x) + \alpha\int v(f(x,u))\,d\mu(x,u)\quad &\forall\, v\in C(X) \\
&&& \int w(x)\,d\mu_0(x) + \int w(x)\,d\hat{\mu}_0(x) = \int w(x)\,d\lambda(x) & \forall\, w\in C(X)\\
&&& \mu\geq 0, \: \mu_0\geq 0,\: \hat{\mu}_0\ge0\\
&&& \mathrm{spt}\:\mu \subset X\times U, \:\: \mathrm{spt}\:\mu_0 \subset X, \:\: \mathrm{spt}\: \hat{\mu}_0 \subset X,
\end{array}
\end{equation}
where the supremum is over the vector of measures $(\mu,\mu_0,\hat{\mu}_0)\in M(X\times U)\times M(X)\times M(X)$.

This is an infinite-dimensional LP in the cone of nonnegative Borel measures. The following Lemma, which is our main theoretical result, relates an optimal solution of this LP to the MCI set $X_I$.

\begin{theorem} 
The optimal value of LP problem (\ref{rrlpd}) is equal to the volume of the MCI set $X_I$, that is, $p^*=\lambda(X_I)$.
Moreover, the supremum is attained by the restriction of the Lebesgue measure to the MCI set $X_I$.
\end{theorem}
\begin{proof}
The proof follows from Lemma~\ref{lem:corrDisc} by the same arguments as Theorem~1 in \cite{roa}.
By definition of the MCI set $X_I$, for any initial condition $x_0\in X_I$ there exists an admissible control sequence such that the associated state sequence remains in $X$. Therefore for any initial measure $\mu_0 \le \lambda$ with $\mathrm{spt}\,\mu_0 \subset X_I$ there exist a discounted occupation measure $\mu$ with $\mathrm{spt}\,\mu\subset X\times U$ and a slack measure $\hat{\mu}_0$ with $\mathrm{spt}\,\hat{\mu}_0 \subset X$ such that the constraints of problem~(\ref{rrlpd}) are satisfied. One such measure $\mu_0$ is the restriction of the Lebesgue measure to $X_I$, and therefore $p^*\ge \lambda(X_I)$. The fact $p^*\le \lambda(X_I)$ follows from Lemma~\ref{lem:corrDisc}.
\end{proof}

\subsection{Continuous time}

The primal LP in continuous time reads
\begin{equation}\label{rrlpc}
\begin{array}{rclll}
p^* & = & \sup & \mu_0(X) \\
&& \mathrm{s.t.} & \beta\int v(x)\, d\mu(x,u) = \int v(x)\,d\mu_0(x) + \int \mathrm{grad}\, v\cdot f(x,u)\,d\mu(x,u)\quad & \forall v\in C^1(X) \\
&&& \int w(x)\,d\mu_0(x) + \int w(x)\,d\hat{\mu}_0(x) = \int w(x)\,d\lambda(x) & \forall\, w\in C(X)\\
&&& \mu\geq 0, \: \mu_0\geq 0,\: \hat{\mu}_0\ge0\\
&&& \mathrm{spt}\:\mu \subset X\times U, \:\: \mathrm{spt}\:\mu_0 \subset X, \:\: \mathrm{spt}\: \hat{\mu}_0 \subset X,
\end{array}
\end{equation}
where the infimum is over the vector of measures $(\mu,\mu_0,\hat{\mu}_0)\in M(X\times U)\times M(X)\times M(X)$.

This is an infinite-dimensional LP in the cone of nonnegative Borel measures. The following Lemma, which is our main theoretical result, relates an optimal solution of this LP to the MCI set $X_I$.

\begin{theorem} 
The optimal value of LP problem (\ref{rrlpc}) is equal to the volume of the MCI set $X_I$, that is, $p^*=\lambda(X_I)$.
Moreover, the supremum is attained by the restriction of the Lebesgue measure to the MCI set $X_I$.
\end{theorem}
\begin{proof}
The fact that $\mu_0$ equal to the restriction of the Lebesgue measure to $X_I$ is feasible in~(\ref{rrlpc}) (and therefore $p^*\ge \lambda(X_I)$) follows by the same arguments as in discrete time. The fact that $p^*\le \lambda(X_I)$ follows from Lemma~\ref{lem:corrCont}.
\end{proof}

\section{Dual LP}\label{sec:dualLP}
In this section we derive LPs dual to the primal LPs~(\ref{rrlpd}) and (\ref{rrlpc}). Since the primal LPs are in the space of measures, the dual LPs will be on the space of continuous functions. Super-level sets of feasible solutions to these LPs then provide outer approximations to the MCI sets, both in discrete and in continuous time. Both duals can be derived by standard infinite-dimensional LP duality theory; see~\cite{roa} for a derivation in a similar setting or~\cite{anderson} for a general theory of infinite-dimensional linear programming.

\subsection{Discrete time}

The dual LP in discrete time reads
\begin{equation}\label{vlpd}
\begin{array}{rclll}
d^* & = & \inf & \displaystyle\int_{X} w(x)\, d\lambda(x) \\
&& \mathrm{s.t.} & \alpha v(f(x,u)) \leq v(x), \:\: &\forall\, (x,u) \in X\times U \\
&&& w(x) \ge v(x) + 1, \:\: &\forall\, x \in X \\
&&& w(x) \geq 0, \:\: &\forall\, x \in X,
\end{array}
\end{equation}
where the infimum is over the pair of functions $(v,w)\in C(X)\times C(X)$.

The following key observation shows that the unit super-level set of any function $w$ feasible in~(\ref{vlpd}) provides an outer-approximation to $X_I$.

\begin{lemma}\label{lem:v0d}
Any feasible solution to problem~(\ref{vlpd}) satisfies $v \ge 0$ and $w\ge 1$ on $X_I$.
\end{lemma}
\begin{proof}
Given any $x_0\in X_I$ there exists a sequence $\{u_t\}_{t=0}^\infty$, $u_t\in U$, such that $x_t\in X$ for all $t$. The first constraint of problem~(\ref{vlpd}) is equivalent to $\alpha v(x_{t+1})\le v(x_t)$, $t\in\{0,1,\ldots\}$. By iterating this inequality we get
\[
v(x_0) \ge \alpha^t v(x_t) \to 0 \quad \text{as}\quad t\to\infty
\]
since $x_t\in X$ and $X$ is bounded. Therefore $v(x_0)\ge0$ and $w(x_0)\ge 1$ for all $x_0\in X_I$.
\end{proof}

The following theorem is instrumental in proving the convergence results of Section~\ref{sec:LMIrelax}.

\begin{theorem}\label{noGapInfd} 
There is no duality gap between primal LP problems (\ref{rrlpd}) on measures
and dual LP problem (\ref{vlpd}) on functions in the sense that
$p^*=d^*$.
\end{theorem}
\begin{proof}
Follows by the same arguments as Theorem~2 in~\cite{roa} using standard infinite-dimensional LP duality theory (see, e.g., \cite{anderson}) and the fact that the feasible set of the primal LP is nonempty and bounded in the metric inducing the weak-* topology on $M(X)\times M(X\times U)\times M(X)$. To see non-emptiness, notice that the vector of measures $(\mu_0,\mu, \hat{\mu}_0) = (0,0,\lambda)$ is trivially feasible. To see the boundedness, it suffices to evaluate the equality constraints of~(\ref{rrlpd}) for $v(x) = w(x) = 1$. This gives $\mu_0(X)  + \hat{\mu}_0(X) = \lambda(X) < \infty$ and $\mu(X) = \mu_0(X)/(1-\alpha)$, which, since $\alpha\in (0,1)$ and all measures are nonnegative, proves the assertion.
\end{proof}

\subsection{Continuous time}

The dual LP in continuous time reads
\begin{equation}\label{vlpc}
\begin{array}{rclll}
d^* & = & \inf & \displaystyle\int_{X} w(x)\, d\lambda(x) \\
&& \mathrm{s.t.} & \mathrm{grad}\,v\cdot f(x,u) \leq \beta v(x), \:\: &\forall\, (x,u) \in X\times U \\
&&& w(x) \ge v(x) + 1, \:\: &\forall\, x \in X \\
&&& w(x) \geq 0, \:\: &\forall\, x \in X,
\end{array}
\end{equation}
where the infimum is over the pair of functions $(v,w)\in C^1(X)\times C(X)$.

The following key observation shows that the unit super-level set of any function $w$ feasible in~(\ref{vlpc}) provides an outer-approximation to $X_I$.

\begin{lemma}\label{lem:v0c}
Any feasible solution to problem~(\ref{vlpc}) satisfies $v\ge 0$ and $w\ge 1$ on $X_I$.
\end{lemma}
\begin{proof}
Given any $x_0\in X_I$ there exists an admissible relaxed control function $\nu_t(\cdot)$, $\nu_t(U) = 1$, such that $x(t)\in X$ for all $t$. For that $x(t)$ we have $\frac{d}{dt}v(x(t)) = \int_U\mathrm{grad}\,v\cdot f(x(t),u)\,d\nu_t(u) \le \int_U \beta v(x(t))\,d\nu_t(u) = \nu_t(U)\beta v(x(t)) = \beta v(x(t))$. Then by Gronwall's inequality $v(x(t)) \le e^{\beta t}v(x_0)$, and consequently
\[
v(x_0) \ge e^{-\beta t}v(x(t)) \to 0\quad \text{as}\quad t\to\infty
\]
since $x(t)\in X$ and $X$ is bounded. Therefore $v(x_0) \ge 0$ and $w(x_0)\ge 1$ for all $x_0\in X_I$.
\end{proof}

The following theorem is instrumental in proving the convergence results of Section~\ref{sec:LMIrelax}.

\begin{theorem}\label{noGapInfc}
There is no duality gap between primal LP problems (\ref{rrlpc}) on measures
and dual LP problem (\ref{vlpc}) on functions in the sense that
$p^*=d^*$.
\end{theorem}
\begin{proof}
Follows by the same arguments as Theorem~2 in~\cite{roa} using standard infinite-dimensional LP duality theory (see, e.g., \cite{anderson}) and the fact that the feasible set of the primal LP is nonempty and bounded in the metric inducing the weak-* topology on $M(X)\times M(X\times U)\times M(X)$. To see non-emptiness, notice that the vector of measures $(\mu_0,\mu, \hat{\mu}_0) = (0,0,\lambda)$ is trivially feasible. To see the boundedness, it suffices to evaluate the equality constraints of~(\ref{rrlpc}) for $v(x) = w(x) = 1$. This gives $\mu_0(X)  + \hat{\mu}_0(X) = \lambda(X) < \infty$ and $\mu(X) = \mu_0(X)/\beta$, which, since $\beta > 0$ and all measures are nonnegative, proves the assertion.
\end{proof}

\section{LMI relaxations}\label{sec:LMIrelax}
In this section we present finite dimensional relaxations of the infinite-dimensional LPs. Both in continuous and discrete time, the relaxations of the primal LPs lead to a truncated moment problem which translates to a semidefinite program (SDP) that can be solved by freely available software,
e.g., SeDuMi~\cite{sedumi} or SDPA \cite{sdpa}. Dual to the primal SDP relaxation is a sum-of-squares (SOS) problem that again translates to an SDP problem.  The following discussion closely follows the one in~\cite{roa_inner}.

We only highlight the main ideas behind the derivation of the finite-dimensional relaxations. The reader is referred to~\cite[Section~5]{roa} or to the comprehensive reference~\cite{lasserre} for details. First, since the supports of all measures feasible in~(\ref{rrlpd}) and~(\ref{rrlpc}) are compact, these measures are uniquely determined by their moments, i.e., by integrals of all monomials (which is a sequence of real numbers when indexed in, e.g., the canonical monomial basis). Therefore, it suffices to restrict the test functions $w(x)$ and $v(x)$ in (\ref{rrlpd}) and (\ref{rrlpc}) to all monomials, reducing the linear equality constraints on measures $\mu_0$,  $\mu$ and $\hat{\mu}_0$ of~(\ref{rrlpd}) and (\ref{rrlpc}) to linear equality constraints on their moments. Next, by the Putinar Positivstellensatz (see~\cite{lasserre,putinar}), the constraint that the support of a measure is included in a given compact basic semialgebraic set is equivalent to the feasibility of an infinite sequence of LMIs involving the so-called moment and localizing matrices, which are linear in the coefficients of the moment sequence. By truncating the moment sequence and taking only the moments corresponding to monomials of total degree less than or equal to $2k$, where $k \in\{1,2,\ldots\}$ is the relaxation order, we obtain a necessary condition for this truncated moment sequence to be the first part of a moment sequence corresponding to a measure with the desired support.
 
 In what follows, $\mathbb{R}_k[\cdot]$ denotes the vector space of real multivariate polynomials of total degree less than or equal to $k$. Furthermore, throughout the rest of this section we make the following standard standing assumption:
 \begin{assumption}\label{compact}
One of the polynomials modeling the sets $X$ resp. $U$
is equal to ${{g_X}_i}(x) = R_X^2-\|x\|^2_2$ resp. ${{g_U}_i}(u) = R_U^2-\|u\|^2_2$ with $R_X$, $R_U$ sufficiently large constants.
\end{assumption}

This assumption is completely without loss of generality since redundant ball constraints can be always added to the description of the \emph{compact} sets $X$ and $U$.

\subsection{Discrete time}

The primal relaxation of order $k$ in discrete time reads
\begin{equation}\label{plmid}
\begin{array}{rclllr}
p^*_k & =  &\max & (y_0)_0 \\
&& \mathrm{s.t.} & A_k(y,y_0,\hat{y}_0) = b_k \\
&&& M_k(y) \succeq 0, & M_{k-{d_X}_i}({g_X}_i,y) \succeq 0, \quad &i=1,2,\ldots,n_X \\
&&& &M_{k-{d_U}_i}({g_U}_{i},y) \succeq 0, \quad &i=1,2,\ldots,n_U \\
&&& M_k(y_0) \succeq 0, & M_{k-{d_X}_i}({g_X}_i,y_0) \succeq 0, \quad &i=1,2,\ldots,n_X \\
&&& M_k(\hat{y}_0) \succeq 0,  & M_{k-{d_X}_i}({g_X}_i,\hat{y}_0) \succeq 0, \quad &i=1,2,\ldots,n_X, 
\end{array}
\end{equation}
where the notation $\succeq 0$ stands for positive semidefinite
and the minimum is over moment sequences $(y, y_0, \hat{y}_0)$ truncated to degree $2k$ corresponding to measures $\mu$, $\mu_0$ and $\hat{\mu}_0$ in~(\ref{rrlpd}). The linear equality constraint captures the two linear equality constraints of~(\ref{rrlpd}) with $v(t,x)\in\mathbb{R}_{2k}[t,x]$ and $w(x)\in\mathbb{R}_{2k}[x]$ being monomials of total degree less than or equal to $2k$. The matrices $M_k(\cdot)$ are the moment and localizing matrices, following the notations of \cite{lasserre} or \cite{roa}.
In problem~(\ref{plmid}), a linear objective is minimized subject to linear equality constraints and LMI constraints; therefore problem (\ref{plmid}) is a semidefinite program (SDP).

The dual relaxation of order $k$ in discrete time reads
\begin{equation}\label{dlmid}
\begin{array}{rcll}
d^*_k & = & \inf & w'l \\\vspace{0.5mm}
&& \mathrm{s.t.} & v(x) - \alpha v(f(x,u)) = q_0(x,u) + \sum_{i=1}^{n_X} q_i(x,u) {g_X}_i(x) + \sum_{i=1}^{n_U} r_i(x,u) {g_U}_i(u) \\ \vspace{1mm}
&&& w(x)-v(x)-1 = p_0(x) + \sum_{i=1}^{n_X} p_i(x) {g_X}_i(x) \\
&&& w(x) = s_0(x) + \sum_{i=1}^{n_X} s_i(x) {g_X}_i(x),
\end{array}
\end{equation}

{\vspace{1mm}}

where $l$ is the vector of Lebesgue moments over $X$ indexed in the same basis in which the polynomial $w(x)$ with coefficients $w$ is expressed. The minimum is over polynomials $v(x) \in {\mathbb R}_{2k}[x]$ and $w\in \mathbb{R}_{2k}[x]$,
and polynomial sum-of-squares $q_i$, $p_i$, $s_i$, $i = 1,\ldots,n_X$ and $r_i$, $i=1,\ldots,n_U$, of appropriate degrees.
In problem~(\ref{dlmid}), a linear objective function is minimized subject to sum-of-squares (SOS) constraints; therefore problem~(\ref{dlmid}) is an SOS problem which can be readily cast as an SDP (see, e.g., \cite{lasserre}).

\subsection{Continuous time}
The primal relaxation of order $k$ in continuous time reads
\begin{equation}\label{plmic}
\begin{array}{rclllr}
p^*_k & =  &\max & (y_0)_0 \\
&& \mathrm{s.t.} & A_k(y,y_0,\hat{y}_0) = b_k \\
&&& M_k(y) \succeq 0, & M_{k-{d_X}_i}({g_X}_i,y) \succeq 0, \quad &i=1,2,\ldots,n_X \\
&&& &M_{k-{d_U}_i}({g_U}_{i},y) \succeq 0, \quad &i=1,2,\ldots,n_U \\
&&& M_k(y_0) \succeq 0, & M_{k-{d_X}_i}({g_X}_i,y_0) \succeq 0, \quad &i=1,2,\ldots,n_X \\
&&& M_k(\hat{y}_0) \succeq 0,  & M_{k-{d_X}_i}({g_X}_i,\hat{y}_0) \succeq 0, \quad &i=1,2,\ldots,n_X,
\end{array}
\end{equation}
where the notation $\succeq 0$ stands for positive semidefinite
and the minimum is over moment sequences $(y, y_0, \hat{y}_0)$ truncated to degree $2k$ corresponding to measures $\mu$, $\mu_0$ and $\hat{\mu}_0$ in~(\ref{rrlpc}). The linear equality constraint captures the two linear equality constraints of~(\ref{rrlpc}) with $v(t,x)\in\mathbb{R}_{2k}[t,x]$ and $w(x)\in\mathbb{R}_{2k}[x]$ being monomials of total degree less than or equal to $2k$. The matrices $M_k(\cdot)$ are the moment and localizing matrices, following the notations of \cite{lasserre} or \cite{roa}.
In problem~(\ref{plmic}), a linear objective is minimized subject to linear equality constraints and LMI constraints; therefore problem (\ref{plmic}) is a semidefinite program (SDP).

The dual relaxation of order $k$ in continuous time reads
\begin{equation}\label{dlmic}
\begin{array}{rcll}
d^*_k & = & \inf & w'l \\\vspace{0.5mm}
&& \mathrm{s.t.} & \beta v(x) - \mathrm{grad}\,v\!\cdot\! f(x,u) = q_0(x,u)\!+\!\sum_{i=1}^{n_X} q_i(x,u) {g_X}_i(x)\!+\! \sum_{i=1}^{n_U} r_i(x,u) {g_U}_i(u) \\ \vspace{1mm}
&&& w(x)-v(x)-1 = p_0(x) + \sum_{i=1}^{n_X} p_i(x) {g_X}_i(x) \\
&&& w(x) = s_0(x) + \sum_{i=1}^{n_X} s_i(x) {g_X}_i(x),
\end{array}
\end{equation}

{\vspace{1mm}}

where $l$ is the vector of Lebesgue moments over $X$ indexed in the same basis in which the polynomial $w(x)$ with coefficients $w$ is expressed. The minimum is over polynomials $v(x) \in {\mathbb R}_{2k}[x]$ and $w\in \mathbb{R}_{2k}[x]$,
and polynomial sum-of-squares $q_i$, $p_i$, $s_i$, $i = 1,\ldots,n_X$ and $r_i$, $i=1,\ldots,n_U$, of appropriate degrees.
In problem~(\ref{dlmic}), a linear objective function is minimized subject to sum-of-squares (SOS) constraints; therefore problem~(\ref{dlmic}) is an SOS problem which can be readily cast as an SDP (see, e.g., \cite{lasserre}).

\subsection{Convergence results}
In this section we state several convergence results for the finite dimensional relaxations resp. approximations (\ref{plmid}), (\ref{plmic}) resp. (\ref{dlmid}), (\ref{dlmic}). Let $w_k$ and $v_k$ denote an optimal solution to the $k^{\mathrm{th}}$ dual SDP approximation~(\ref{dlmid}) or (\ref{dlmic}), and define
\[
{X_{I}}_k := \{x\in X \: :\: v_k(x) \ge 0 \}.
\]
Then, in view of Lemmata~\ref{lem:v0d} and \ref{lem:v0c}, we know that $w_k$ over-approximates the indicator function of the MCI set $X_I$ on $X$, i.e., $w_k \ge I_{X_I}$ on $X$, and that the sets ${X_I}_k$ approximate from the outside the MCI set $X_I$, i.e., ${X_I}_k \supset X_I$. In the sequel we prove the following:
\begin{itemize}
	\item The optimal values of the finite-dimensional primal and dual problems $p_k^*$ and $d_k^*$ coincide and converge to the optimal values of the infinite dimensional primal and dual LPs $p^*$ and $ d^*$ which also coincide (in view of Theorems~\ref{noGapInfd} and \ref{noGapInfc}) and are equal to the volume of the MCI set.
	\item The sequence of functions $w_k$ converges on $X$ from above to the indicator function of the MCI set in $L_1$ norm. In addition, the running minimum $\min_{i\le k} w_i$ converges on $X$ from above to the indicator function of the MCI set set in $L_1$ norm and almost uniformly.
	\item The sequence of sets ${X_I}_k$ converges to the MCI set $X_I$ in the sense that the volume discrepancy tends to zero, i.e., $\lim_{k\to \infty} \lambda(X_{Ik} \setminus X_I) = 0$.
\end{itemize}

The proofs of the results follow very similar reasoning as analogous results on region of attraction approximations in~\cite[Section~6]{roa}.

\begin{lemma}\label{lem:noGapRelax}
There is no duality gap between primal LMI problems (\ref{plmid} and \ref{plmic}) and dual LMI problems (\ref{dlmid} and \ref{dlmic}),
i.e. $p^*_k = d^*_k$.
\end{lemma}
\begin{proof}
The argument closely follows the one in~\cite[Theorem~4]{roa} and therefore we only outline the key points of the proof. To prove the absence of duality gap, it is sufficient to show that the feasible sets of the primal SDPs~(\ref{plmid}) and (\ref{plmic}) are non-empty and compact. The result then follows by standard SDP duality theory (see~\cite[Theorem~4]{roa} for a detailed argument). The non-emptiness follows trivially since the vector of measures $(\mu_0,\mu, \hat{\mu}) = (0,0,\lambda)$ is feasible in the primal infinite-dimensional LPs~(\ref{rrlpd}) and (\ref{rrlpc}) and therefore the truncated moment sequences corresponding to these measures are feasible in the primal SDP relaxations~(\ref{plmid}) and (\ref{plmic}). To see the compactness observe that the first components (i.e., masses) of the truncated moment vectors $y_0$, $y$ and $\hat{y}$ are bounded. This follows by evaluating the equality constraints of~(\ref{rrlpd}) and (\ref{rrlpc}) for $w(x) = v(x) = 1$. Indeed, in discrete-time we get $(y)_0 = (y_0)_0 / (1-\alpha)$ and in continuous-time we get $(y)_0 = (y_0)_0 / \beta$; in addition, in both cases we have $(y_0)_0 + (\hat{y}_0)_0 = \lambda(X) < \infty$ and therefore the first components are indeed bounded (since they are trivially bounded from below, in fact nonnegative, due to the constraints on moment matrices). Boundedness of the even components of each truncated moment
vector then follows from the structure of the localizing matrices corresponding to the functions from Assumption~\ref{compact}. Boundedness of the entire truncated moment vectors then
follows since the even moments appear on the diagonal of the positive semidefinite moment matrices.
\end{proof}

The following result shows the convergence of the optimal values of the relaxations to the optimal values of the infinite-dimensional LPs.
\begin{theorem}\label{thm:pdconv}
The sequence of infima of LMI problems~(\ref{dlmid}) and (\ref{dlmic})  converges monotonically from above to the supremum of the LP problems~(\ref{vlpd}) and (\ref{vlpc}), i.e., $d^*\le d_{k+1}^* \le d_k^*$ and $\lim_{k\to\infty} d_k^* = d^*$. Similarly, the sequence of maxima of LMI problems (\ref{plmid}) and (\ref{plmic}) converges monotonically from above to the maximum of the LP problems (\ref{rrlpd}) and (\ref{rrlpc}), i.e., $p^* \le p^*_{k+1}\le p_k^*$ and $\lim_{k\to \infty} p^*_k = p^*$.
\end{theorem}
\begin{proof}
The monotonicity of the optimal values of the relaxations $p_k^*$ resp. approximations $d_k^*$ is evident form the structure of the feasible sets of the corresponding SDPs. The convergence of the primal relaxations $p_k$ to $p^*$ follows from the compactness of the feasible sets of the primal SDPs (\ref{plmid}) and (\ref{plmic}) (shown in the proof of Lemma~\ref{lem:noGapRelax}) by standard arguments on the convergence of Lasserre's LMI hierarchy (see, e.g., \cite{lasserre}). The converge of the optimal value of the dual approximations $d_k^*$ to $d^*$ then follows from Lemma~\ref{lem:noGapRelax}.
\end{proof}

The next theorem shows functional convergence from above to the indicator function of the MCI set.
\begin{theorem}\label{thm:dualConvFun}
Let $w_k \in {\mathbb R}_{2k}[x]$ denote the $w$-component of a solution to the dual LMI problems (\ref{dlmid}) or (\ref{dlmic}) and let $\bar{w}_k(x) =\min_{i\le k} w_i(x)$. Then $w_k$ converges from above to $I_{X_I}$ in $L^1$ norm and  $\bar{w}_k$ converges from above to $I_{X_I}$ in $L^1$ norm and almost uniformly.
\end{theorem}
\begin{proof}
The convergence in $L_1$ norm follows immediately from Theorem~\ref{thm:pdconv} and from the fact that $w_k \ge I_{X_I}$ by Lemmata~\ref{lem:v0d} and \ref{lem:v0c}. The convergence of the running minima follows from the fact that there exists a subsequence of $\{w_k\}_{k=0}^\infty$ which converges almost uniformly (by, e.g., \cite[Theorems~2.5.2 and ~2.5.3]{ash}).
\end{proof}

Our last theorem shows a set-wise convergence of the outer-approximations to the MCI set.
\begin{theorem}
Let $(v_k,w_k)\in {\mathbb R}_{2k}[x]\times {\mathbb R}_{2k}[x]$ denote an optimal solution to the dual LMI problem (\ref{dlmid}) or (\ref{dlmic}) and let ${X_I}_k := \{x \in {\mathbb R}^n \: :\: v_k(x) \geq 0\}$. Then $X_I \subset {X_I}_k$,
\[
 \lim_{k\to\infty}\lambda(X_{Ik}\setminus X_I) = 0 \:\:\:\text{and}\:\:\:
 \lambda(\cap_{k=1}^{\infty} {X_I}_k \setminus X_I) = 0.
\]
\end{theorem}
\begin{proof} From Lemmata~\ref{lem:v0d} or \ref{lem:v0c} we have ${X_I}_k\supset X_I$ and $w_k \ge I_{X_I}$; therefore, since $w\ge v +1$ and $w\ge 0$ on $X$, we have $w_k\ge I_{{X_I}_k}\ge I_{X_{I}}$ and $\{x : w_k(x) \ge 1\}\supset  {X_{I}}_k\supset X_0$. From Theorem~\ref{thm:dualConvFun}, we have $w_k \to I_{X_I}$ in $L^1$ norm on $X$. Consequently, 
\begin{align*}
\lambda(X_I) = \int_X I_{X_I}\,d\lambda & = \lim_{k\to\infty} \int_X w_k\,d\lambda  \ge  \lim_{k\to\infty} \int_X I_{{X_I}_k}\,d\lambda\\&   =  \lim_{k\to\infty}\lambda({X_I}_k) \ge \lim_{k\to\infty}\lambda(\cap_{i=1}^k {X_I}_i) = \lambda(\cap_{k=1}^\infty {X_I}_k). 
\end{align*}
But since $X_I \subset {X_I}_k$ for all $k$, we must have
\[
 \lim_{k\to\infty}\lambda({X_I}_k) = \lambda(X_I)\quad \mathrm{and}\quad  \lambda(\cap_{k=1}^\infty {X_I}_k) = \lambda(X_I),
\]
and the theorem follows.
\end{proof}

\section{Numerical examples}\label{sec:NumEx}
In this section we present numerical examples that illustrate our results. The primal SDP relaxations were modeled using Gloptipoly 3~\cite{glopti} and the dual SOS problems using Yalmip~\cite{yalmip}. The resulting SDP problems were solved using SeDuMi~\cite{sedumi} (which, in the case of primal relaxations, also returns the dual solution providing the outer approximations). For numerical computation (especially for higher relaxation orders), the problem data should be scaled such that the constraint sets are (within) unit boxes or unit balls; for ease of reproduction, most of the numerical problems shown are already scaled. On our problem class we observed only marginal sensitivity to the values of the discrete- and continuous-time discount factors $\alpha$ and $\beta$ and report results with $\alpha = 0.9$ and $\beta = 1$ for all examples presented.

For a discussion on the scalability of our approach and the performance of alternative SDP solvers see the Conclusion and the acrobot-on-a-cart example below.

\subsection{Discrete time}

\subsubsection{Double integrator}\label{sec:DoubleIntegDiscrete}

Consider the discrete-time double integrator:
\begin{align*}
	x_1^+ &= x_1 + 0.1x_2   \\
	x_2^+ & = x_2 + 0.05u
\end{align*}
with the state constraint set $X = [-1,1]^2$ and input constraint set $U = [-0.5,0.5]$. The resulting of MCI set outer approximations of degree 8 and 12 are shown in Figure~\ref{fig:1d}; the approximation is fairly tight for modest degrees. The true MCI set was computed using the standard algorithm based on polyhedral projections~\cite{blanchini}.

\subsubsection{Cathala system}

Consider the Cathala system borrowed from~\cite{chatala}:
\begin{align*}
	x_1^+ &= x_1 + x_2   \\
	x_2^+ & = -0.5952 + x_2 + x_1^2.
\end{align*}
The chaotic attractor of this system is contained in the set $X = [-1.6,1.6]^2$. MCI set outer approximations are shown in Figure~\ref{fig:2d}; again, the approximations are relatively tight for small relaxation orders. The true MCI set was (approximately) computed by gridding.

\begin{figure*}[ht]
	\begin{picture}(140,180)
	\put(20,20){\includegraphics[width=70mm]{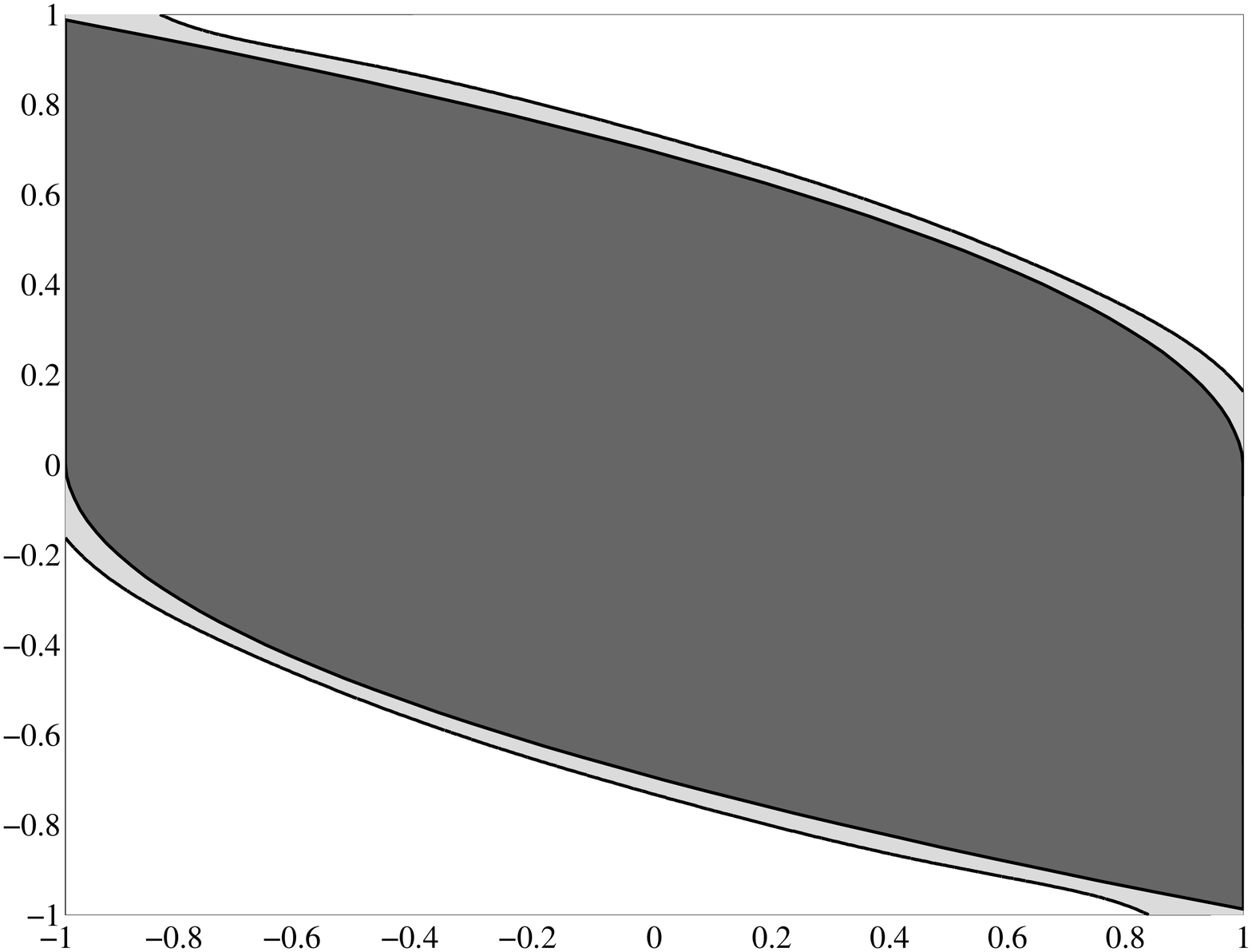}} 
	\put(250,20){\includegraphics[width=70mm]{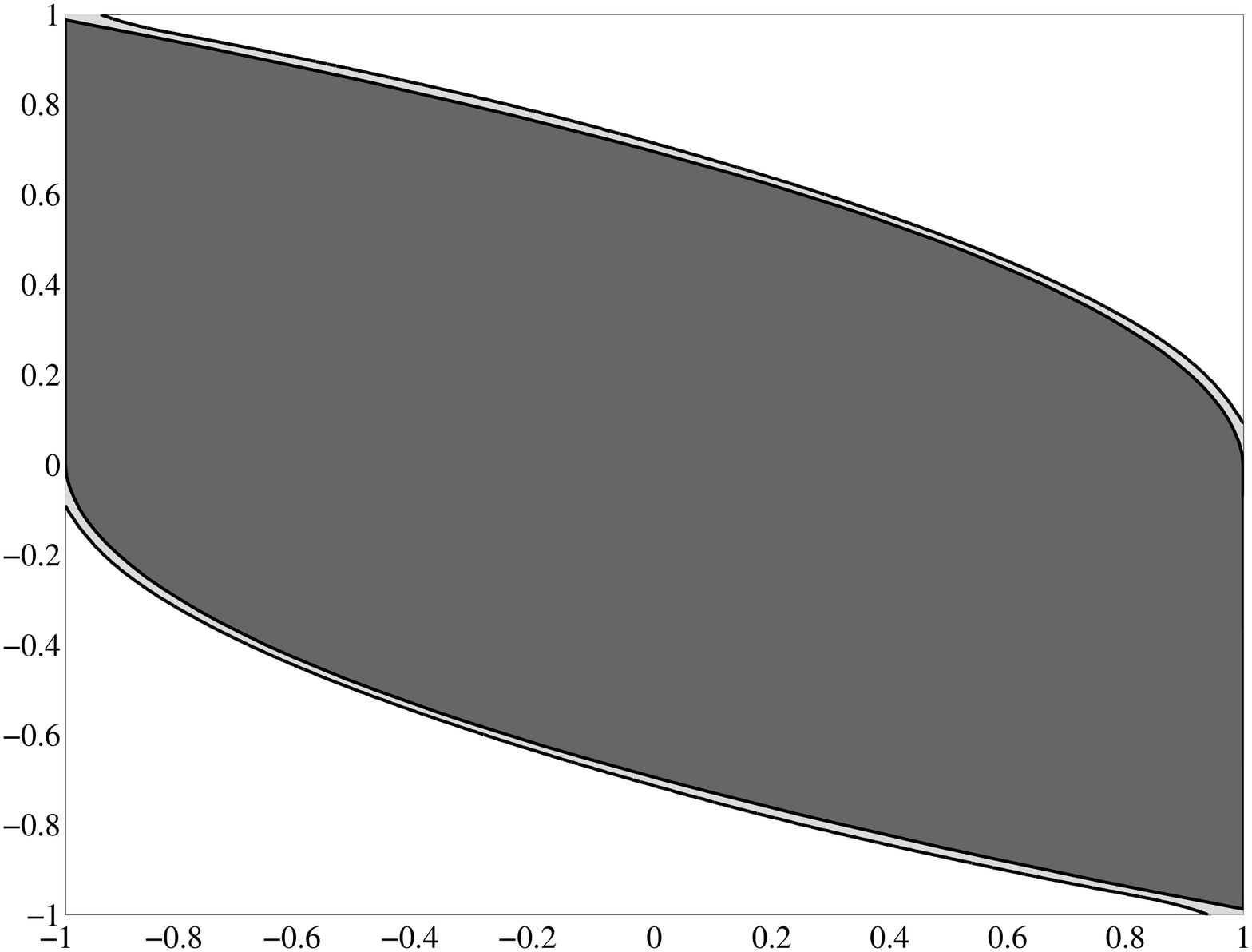}}

	\put(120,10){\small $x$}
	\put(350,10){\small $x$}

	\put(165,150){\small $d = 8$}
	\put(395,150){\small $d=12$}
	\end{picture}
	\caption{Discrete time double integrator -- polynomial outer approximations (light gray) to the MCI set (dark gray) for degrees $d\in\{8,12\}$.}
	\label{fig:1d}
\end{figure*}

\begin{figure*}[ht]
	\begin{picture}(140,180)
	\put(20,20){\includegraphics[width=70mm]{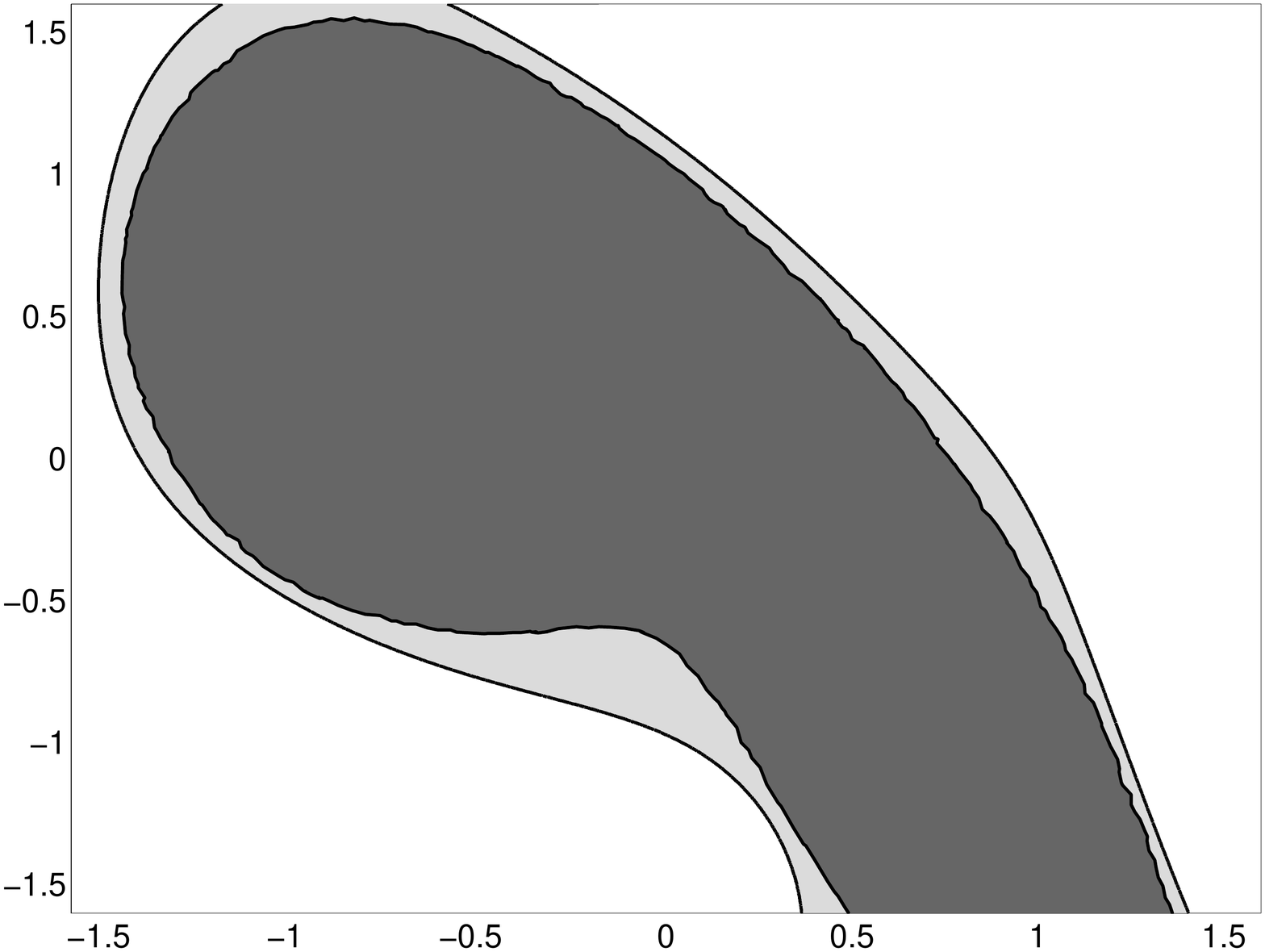}} 
	\put(250,20){\includegraphics[width=70mm]{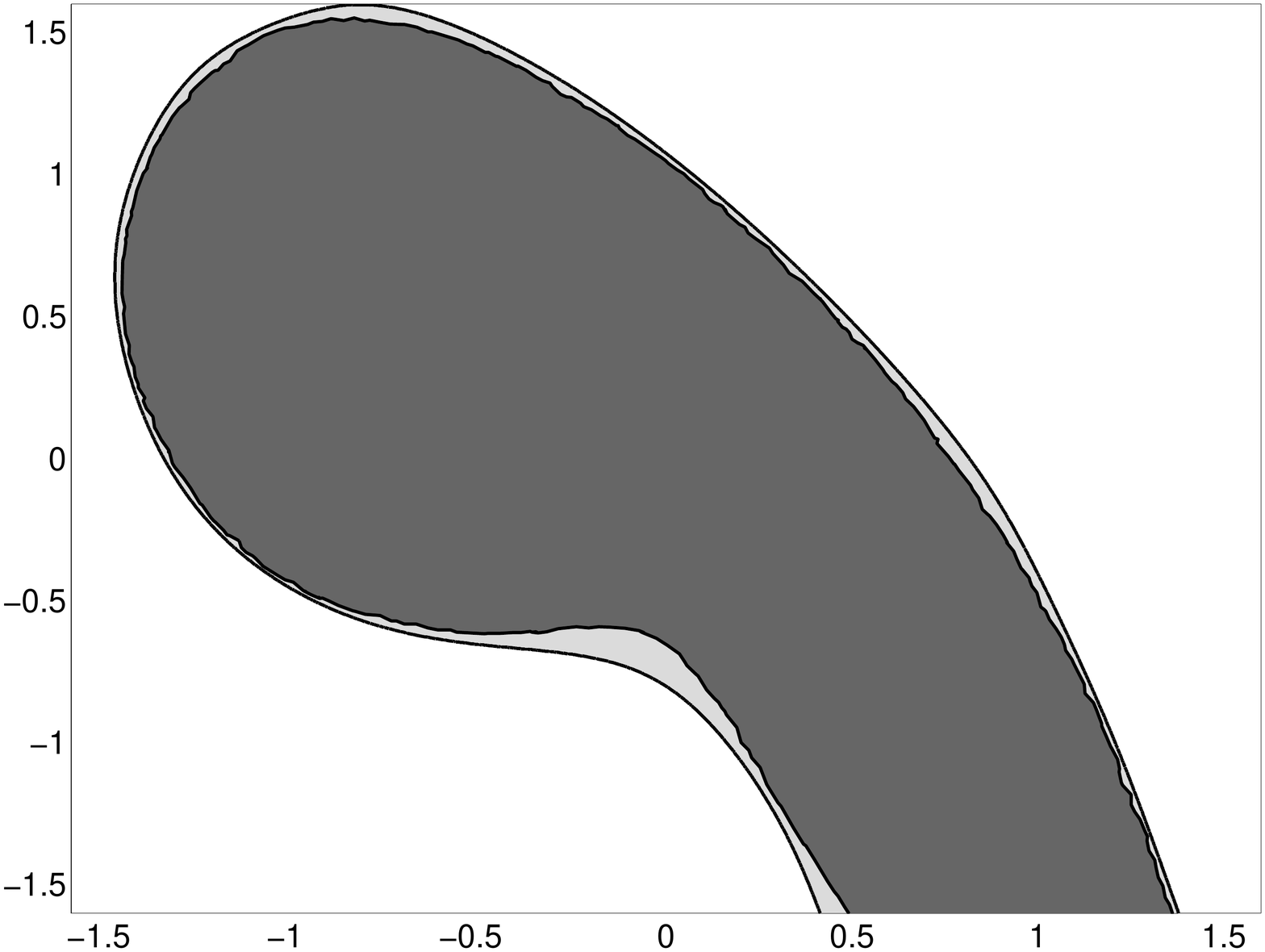}}

	\put(120,10){\small $x$}
	\put(350,10){\small $x$}

	\put(165,150){\small $d = 6$}
	\put(395,150){\small $d=10$}
	\end{picture}
	\caption{\small Cathala system -- polynomial outer approximations (light gray) to the MCI set (dark gray) for degrees $d\in\{6,10\}$.}
	\label{fig:2d}
\end{figure*}

\subsubsection{Julia sets}

Consider over $z \in \mathbb C$, or equivalently over $x \in {\mathbb R}^2$
with $z := x_1+ix_2$, the quadratic recurrence
\begin{align*}
	z^+ &= z^2+c
\end{align*}
with $c \in \mathbb C$ a given complex number and $i$ the imaginary unit.
The filled Julia set is the set of all initial conditions of the above recurrence for which the trajectories remain bounded. The shape of the Julia set depends strongly on the parameter $c$. If $c$ lies inside the Mandelbrot set, then the Julia set is connected; otherwise the set is disconnected. In both cases the boundary of the set has a very complicated (in fact fractal) structure. Here we shall compute outer approximations of the filled Julia set intersected with the unit ball. To this end we set $X = \{x\in\mathbb{R}^2\: :\: \|x\| \le 1\}$. Figure \ref{fig:j1} shows outer approximations
of degree~12 for parameter values $c=-0.7+i0.2$ (inside the Mandelbrot set)
and $c=-0.9+i0.2$ (outside the Mandelbrot set). The ``true'' filled Julia set was (approximately) obtained by randomly sampling initial conditions within the unit ball and iterating the recurrence for one hundred steps.
Taking higher degree of the approximating polynomials does not give significant improvements
due to our choice of the monomial basis to represent
polynomials. An alternative basis (e.g. Chebyshev polynomials --
see the related discussions in \cite{volume} and \cite{roa})
would allow us to improve further the outer estimates
and better capture the intricate structure of the filled Julia set's boundary. 

\begin{figure*}[t!]
	\begin{picture}(140,206)
	\put(5,20){\includegraphics[width=80mm]{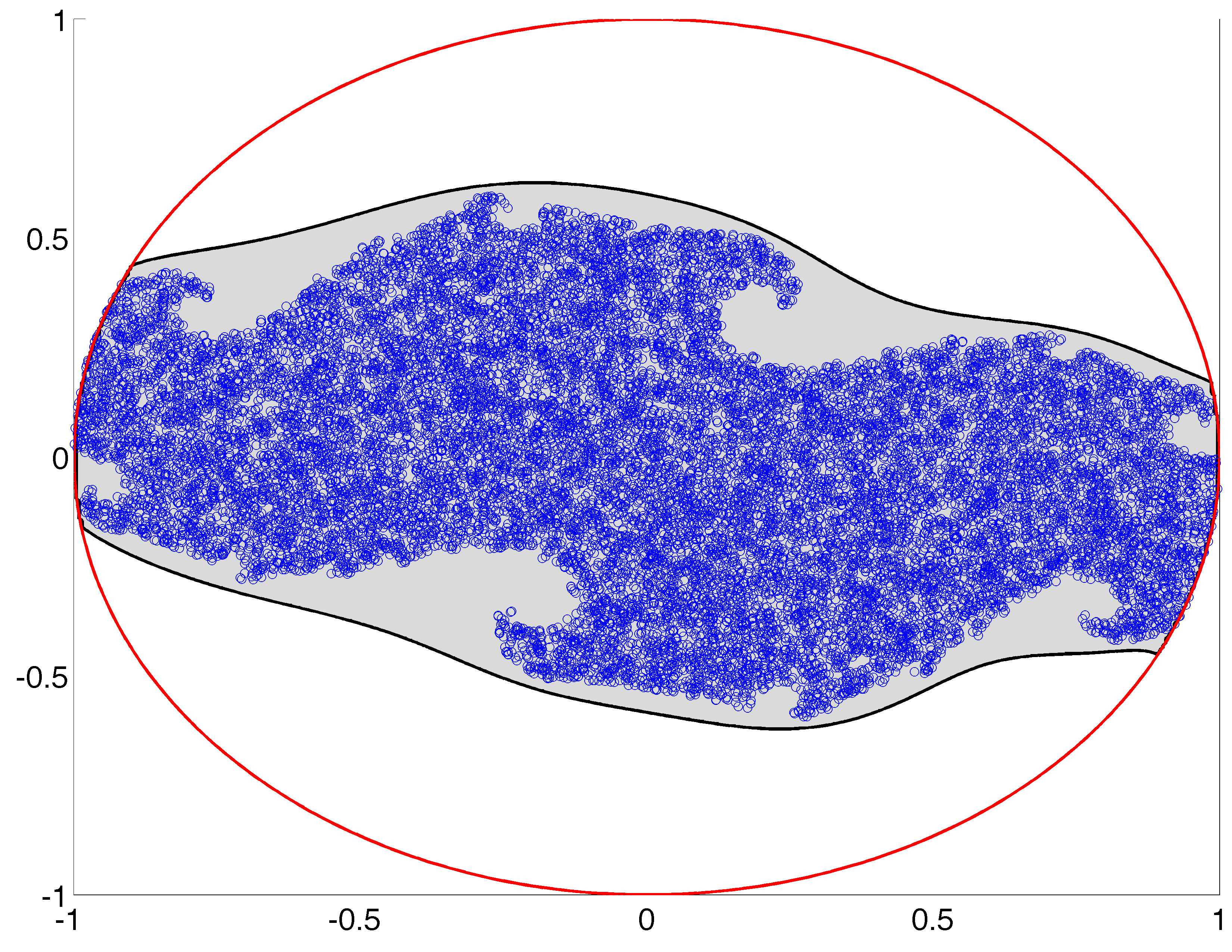}} 
	\put(235,20){\includegraphics[width=80mm]{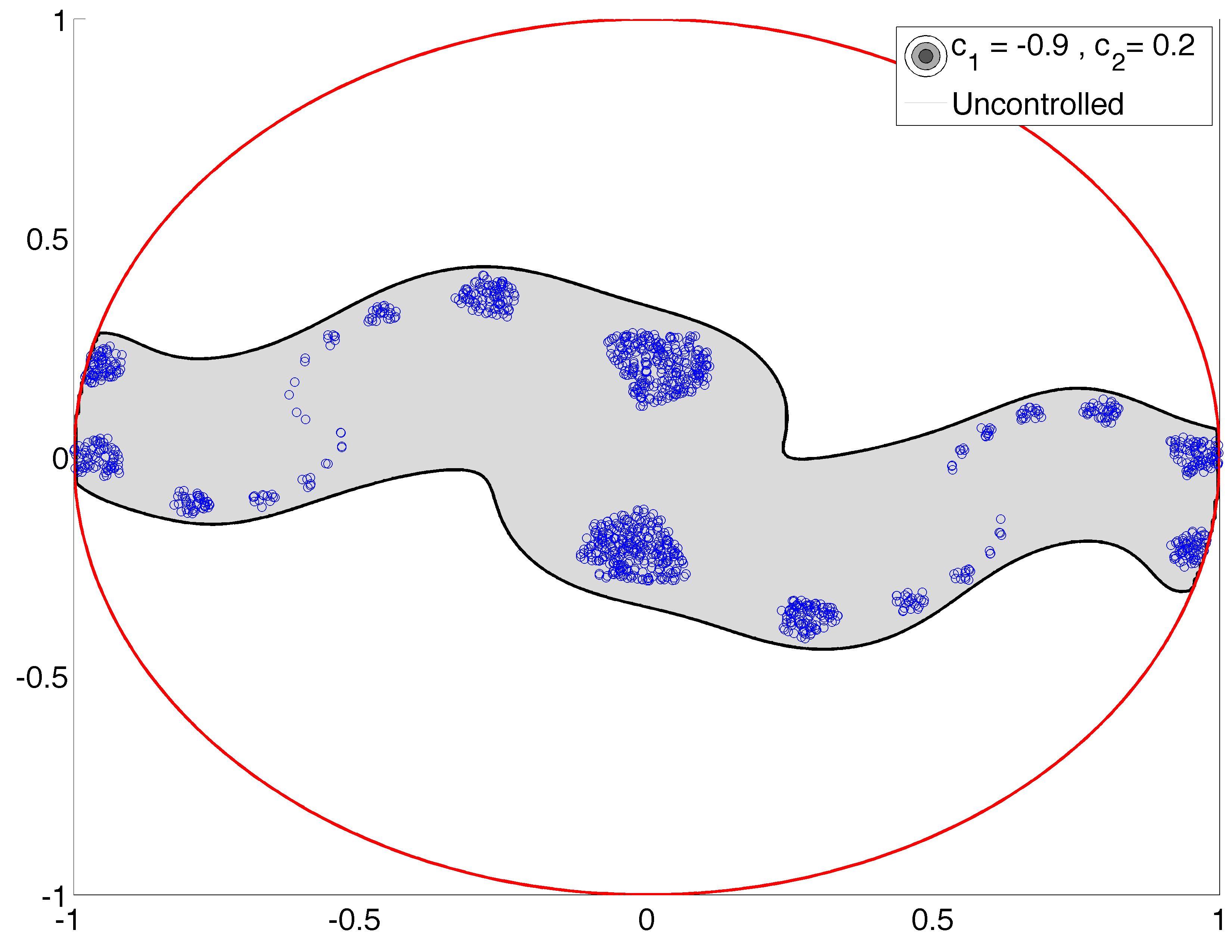}}

	\put(85,10){\small $c=-0.7+i0.2$}
	\put(315,10){\small $c=-0.9+i0.2$}

	\end{picture}
	\caption{\small Filled Julia set -- polynomial outer approximation of degree 12 (light gray) and (an approximation of) the ``true'' set (dark grey) represented as an ensemble of initial conditions randomly sampled within the state-constraint set. The dashed line shows the boundary of the unit-ball state-constraint set.}
	\label{fig:j1}
\end{figure*}

 \subsubsection{H\'enon map}
 Consider the modified controlled H\'enon map
\begin{align*}
 x_1^+ &= 0.44 - 0.1x_3 - 4x_2^2 + 0.25u,\\
 x_2^+ &= x_1-4x_1x_2,\\
 x_3^+ &= x_2,
\end{align*}
adapted from \cite{henon} with $X = [-1,1]^3$ and $U = [-u_\mathrm{max},u_\mathrm{max}]$. We investigate two cases: uncontrolled (i.e., $u_\mathrm{max} = 0$) and controlled with $u_\mathrm{max} = 1$. Figure~\ref{fig:Henon} shows outer approximations to the MCI set of degree eight for both settings and the ``true'' MCI set in the uncontrolled setting (approximately) obtained by random sampling of initial conditions inside the constraint set $X$. The outer approximations suggest that, as expected, allowing for control leads to a larger MCI set. 
 
 \begin{figure*}[ht]
	\begin{picture}(100,265)
	\put(135,20){\includegraphics[width=61mm]{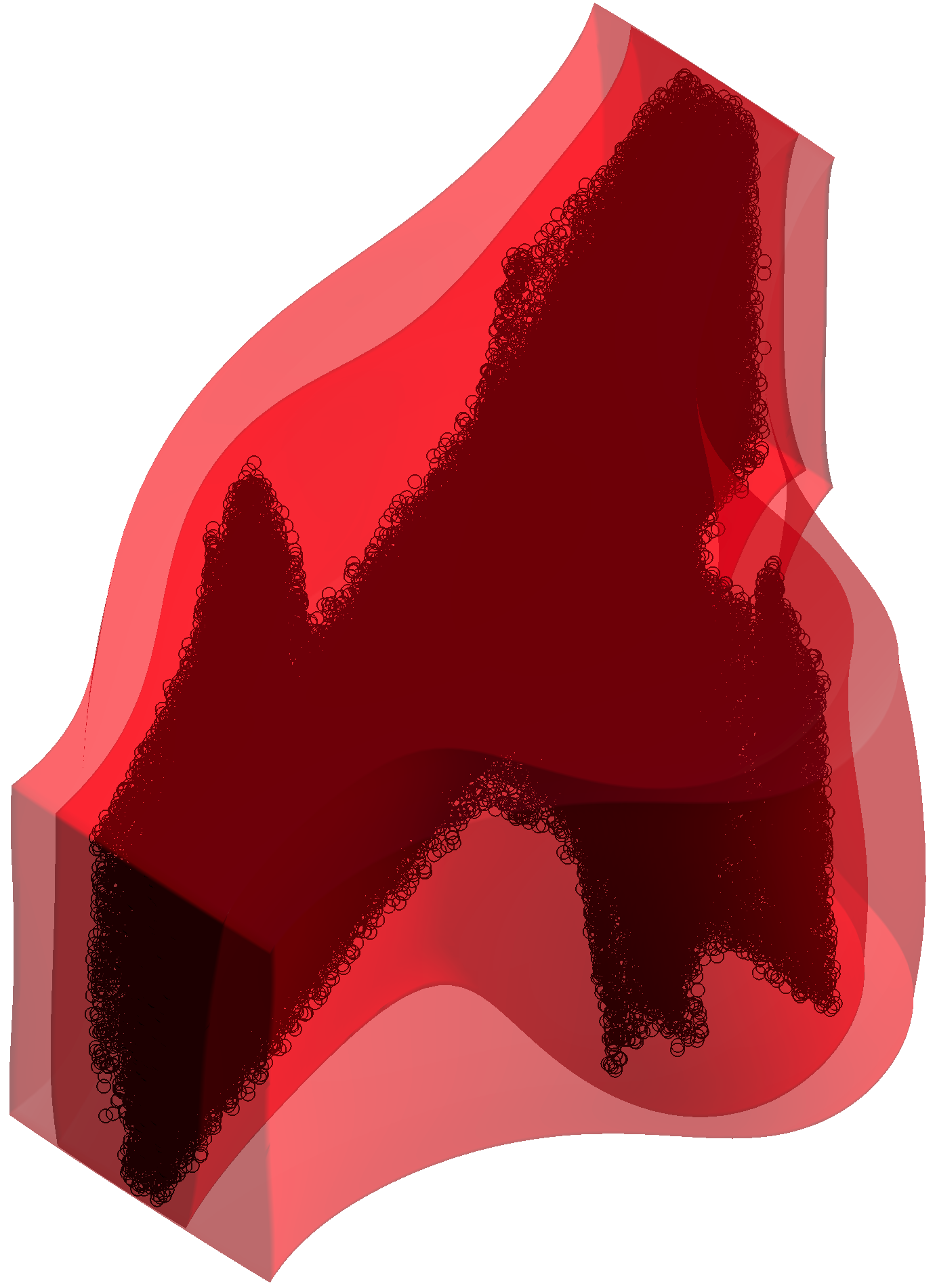}} 
	\end{picture}
	\caption{\small Controlled H\'enon map -- polynomial outer approximation of degree eight in the uncontrolled setting (darker red, smaller) and in the controlled setting (lighter red, larger). The (approximation of) the ``true'' set (black) in the uncontrolled setting is represented as an ensemble of initial conditions randomly sampled within the state-constraint set.}
	\label{fig:Henon}
\end{figure*}

\subsection{Continuous time}

\subsubsection{Double integrator}
Consider the continuous-time double integrator
\begin{align*}
	\dot{x}_1 &= x_2   \\
	\dot{x}_2 & = u,
\end{align*}
with state constraint set $X = [-1,1]^2$ and input constraint set $U = [-1,1]$. The resulting MCI set outer approximations for degrees 8 and 12 are in Figure~\ref{fig:1c}. The approximations are fairly tight even for relatively low relaxation orders. The true MCI set was (approximately) computed as in Section~\ref{sec:DoubleIntegDiscrete} by methods of~\cite{blanchini} after dense time discretization.

\begin{figure*}[ht]
	\begin{picture}(140,180)
	\put(20,20){\includegraphics[width=70mm]{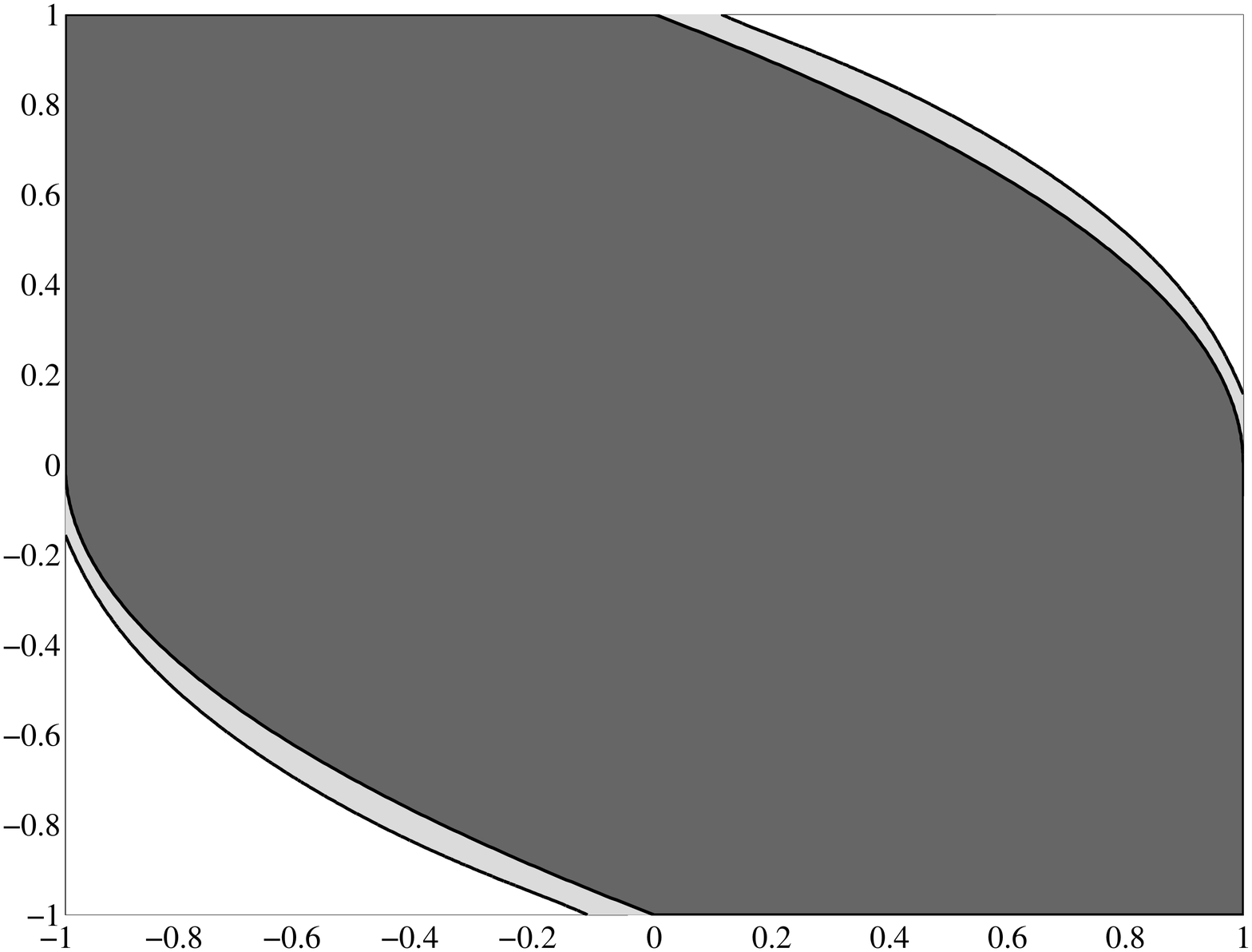}} 
	\put(250,20){\includegraphics[width=70mm]{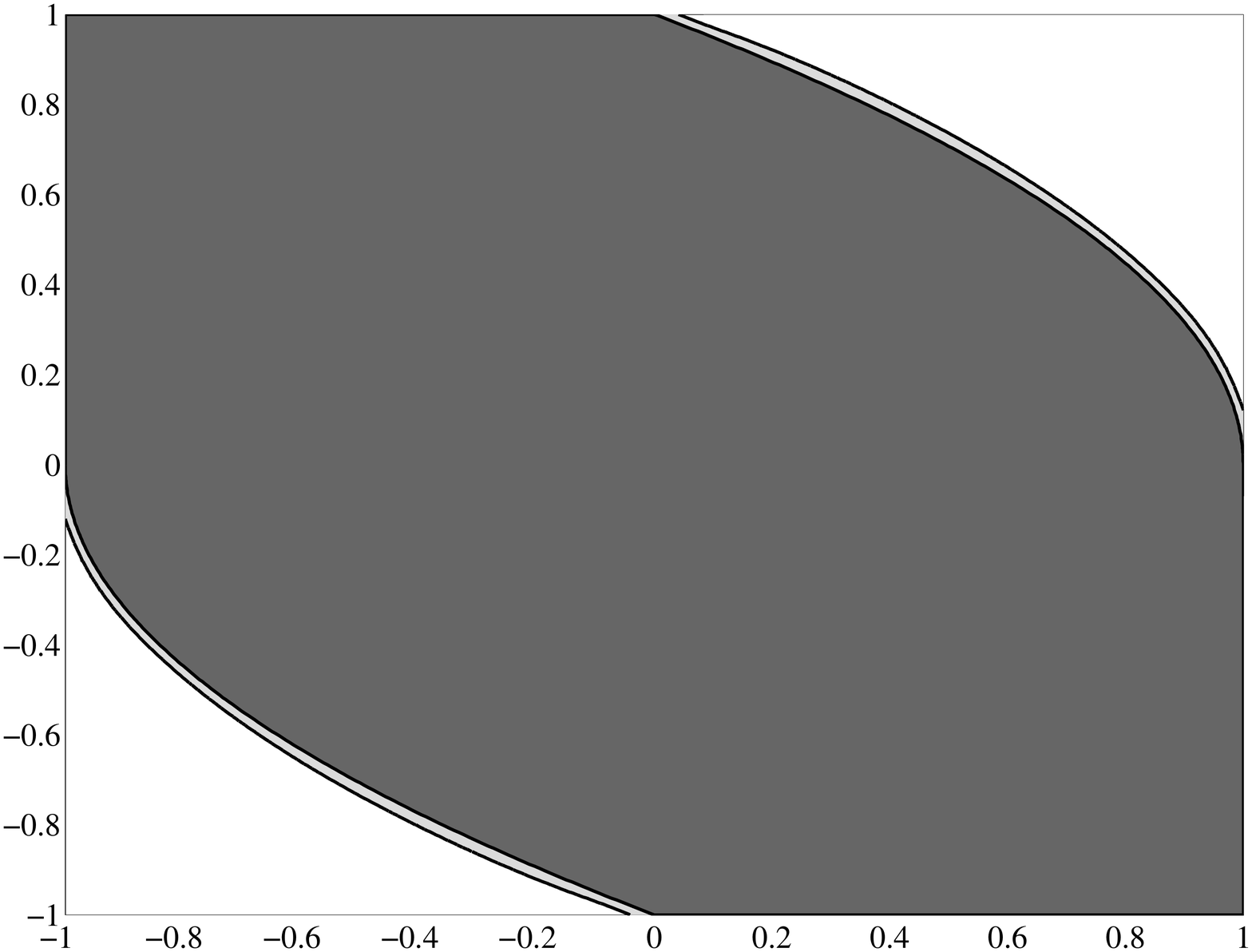}}

	\put(120,10){\small $x$}
	\put(351,10){\small $x$}

	\put(175,160){\small $d = 8$}
	\put(405,160){\small $d=14$}
	\end{picture}
	\caption{\small Continuous-time double integrator -- polynomial outer approximations (light gray) to the MCI set (dark gray) for degrees $d\in\{8,14\}$.}
	\label{fig:1c}
\end{figure*}

\subsubsection{Spider-web system}
As our second example we take the spider-web system from~\cite{ahmadi_msc} given by equations
\begin{align*}
	\dot{x}_1 &= -0.15x_1^7 + 200x_1^6x_2 - 10.5x_1^5x_2^2 - 807x_1^4x_2^3 + 14x_1^3x_2^4 + 600x_1^2x_2^5 - 3.5x_1x_2^6 + 9x_2^7   \\
	\dot{x}_2 & = -9x_1^7 - 3.5x_1^6x_2 - 600x_1^5x_2^2 + 14x_1^4x_2^3 + 807x_1^3x_2^4 - 10.5x_1^2x_2^5 - 200x_1x_2^6 - 0.15x_2^7
\end{align*}
with the constraint set $X = [-1,1]^2$. Here we exploit the fact that the system dynamics are captured by constraints on $v$ only whereas $w$ is merely over approximating $v+1$, and the fact that outer approximations to the MCI set are given not only by $\{x : v(x)\ge 0\}$ but also by $\{x : w(x)\ge 1 \}$. Therefore, if low-complexity outer approximations are desired, it is reasonable to choose different degrees of $v$ and $w$ in~(\ref{dlmic}) -- high for $v$ and lower for $w$ -- and use the set $\{x : w(x)\ge 1\}$ as the outer approximation. That way, we expect to obtain relatively tight low-order approximations. This is confirmed by numerical results shown in Figure~\ref{fig:2c}. The degree of $v$ is equal to 16 for both figures, whereas $\mathrm{deg}\,w = 8$ for the left figure and $\mathrm{deg}\,w = 16$ for the right figure. We observe no significant loss in tightness by choosing a smaller degree of $w$. The true MCI set was (approximately) computed by gridding.

\begin{figure*}[ht]
	\begin{picture}(140,180)
	\put(20,20){\includegraphics[width=70mm]{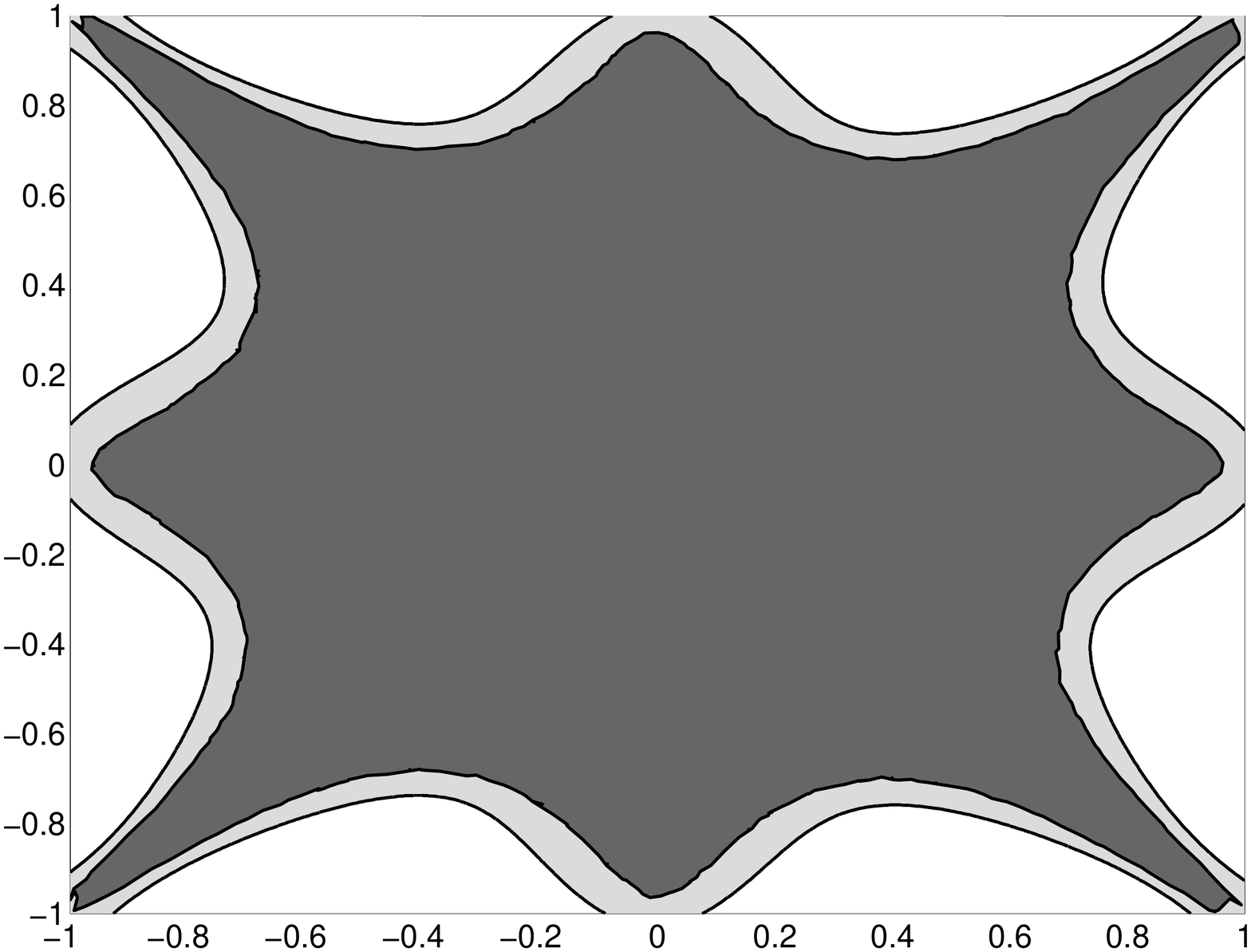}} 
	\put(250,20){\includegraphics[width=70mm]{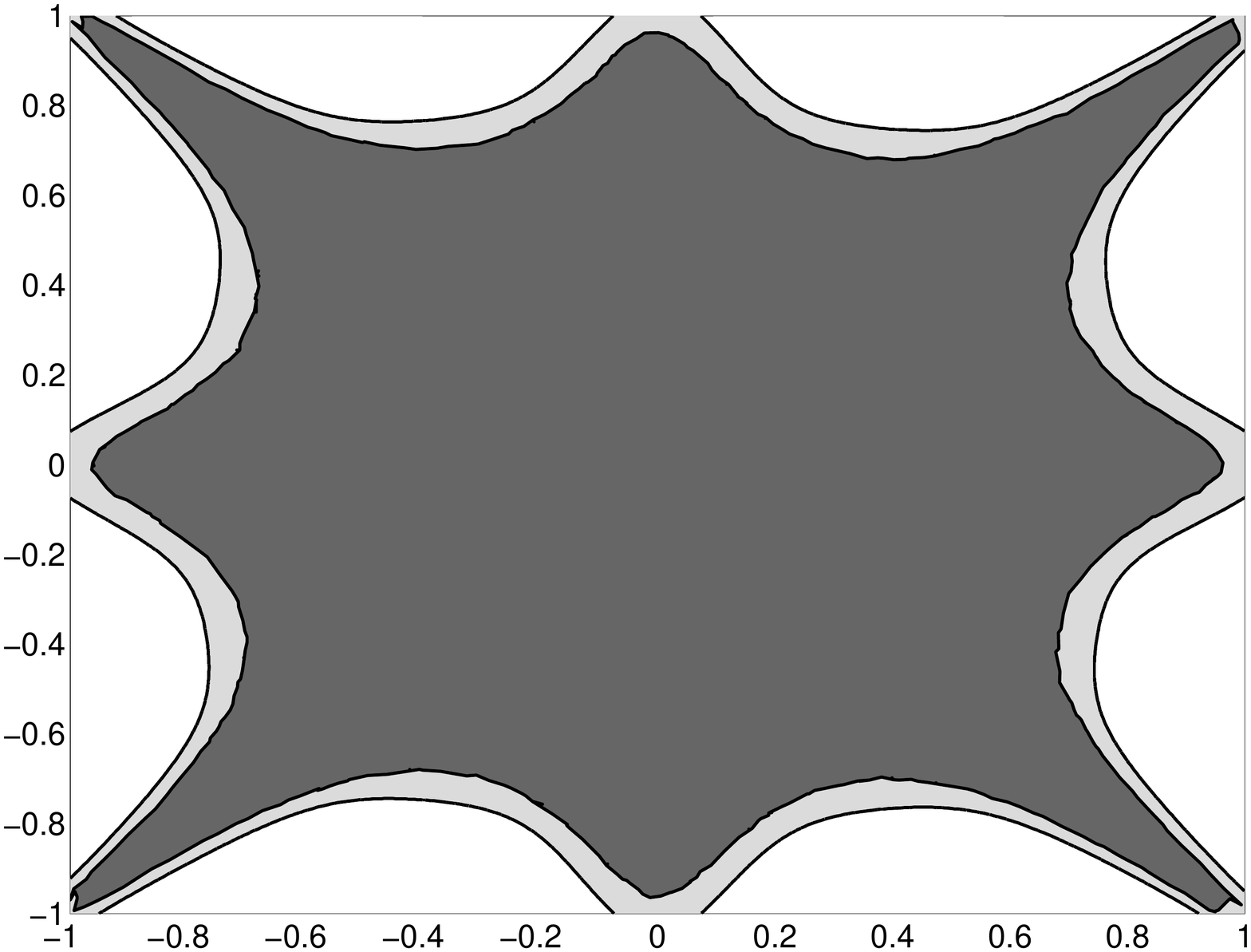}}

	\put(120,10){\small $x$}
	\put(351,10){\small $x$}

	\end{picture}
	\caption{\small Spider-web system -- polynomial outer approximations (light gray) to the MCI set (dark gray) for degrees $\mathrm{deg}\,v = 16$ and $\mathrm{deg}\,w = 8$ on the left and $\mathrm{deg}\,w = 16$ on the right.}
	\label{fig:2c}
\end{figure*}

\subsubsection{Acrobot on a cart}\label{sec:acrobotOnCart}
As our last example we consider the acrobot on a cart system adapted from~\cite{slavka}, which is essentially a double pendulum on a cart where the inputs are the force acting on the cart and the torque in the middle joint of the double pendulum. The system is sketched in Figure~\ref{fig:acrobOnCart_sketch}. It is a sixth order system with with two control inputs; the dynamic equation is given by
\[
\dot{x} = \begin{bmatrix}x_4\\x_5\\x_6\\M(x)^{-1}N(x,u)\end{bmatrix}\in\mathbb{R}^6
\]
where
\[
M(x) = \begin{bmatrix}
			 a_1 		      &    a_2\cos x_2             &    a_3\cos x_3\\
 			a_2\cos x_2   &    a_4 		          &    a_5 \cos(x_2 - x_3) \\
		         a_3\cos x_3    & a_5\cos(x_2 - x_3)    &    a_6
	    \end{bmatrix}
	    \]
	    and
	    \[N(x,u) =
	    \begin{bmatrix}
                     u_1 + a_2x_5^2\sin x_2 + a_3x_6^2\sin x_3 - \delta_0x_4\\
	 	 - a_5x_6^2\sin(x_2 - x_3) + \delta_2x_6 + a_7\sin x_2 - x_5(\delta_1 + \delta_2)\\
	          u_2 + a_5\sin(x_2 - x_3)x_5^2 + \delta_2x_5 - \delta_2 x_6 + a_8\sin x_3\\
	\end{bmatrix}.
\]
The states $x_1$, $x_2$, $x_3$ represent, respectively, the position of the cart (in meters), the angle of the lower rod and the angle of the upper rod of the double pendulum (both in radians); the states $x_4$, $x_5$ and $x_6$ are then the corresponding velocities in meters per second for the cart and radians per second for the pendulum rods. The constants are given by $a_1 = 0.85$, $a_2 = 0.2063$, $a_3 = 0.0688$, $a_4 = 0.0917$, $a_5 = 0.0344$, $a_6 = 0.0229$, $a_7 = 2.0233$, $a_8 =  0.6744$, $\delta_0 = 0.3$, $\delta_1 = 0.1$, $\delta_2 = 0.1$. We are interested in computing the maximum controlled invariant subset of the state constraint set
\[
X = [-1, 1] \times [-\pi/3, \pi/3] \times [-\pi/3, \pi/3] \times [-0.5, 0.5] \times [-5, 5] \times [-5, 5].
\]
We investigate two cases. First, we consider the situation where only the middle joint is actuated and there is no force on the cart; therefore we impose the constraint $(u_1,u_2)\in U= \{0\}\times [-1,1]$. Second, we consider the situation where we can also exert a force on the cart; in this case we impose $(u_1,u_2)\in U= [-1,1] \times [-1,1]$. Naturally, the MCI set for the second case is larger (or at least the same) as for the first case. This is confirmed\footnote{There is no a priori guarantee on set-wise ordering of the outer approximations; what is guaranteed is the ordering of optimal values of the optimization problems~(\ref{plmic}) or (\ref{dlmic}). } by outer approximations of degree four whose section for $x_1 = 0$, $x_4 = 0$, $x_5 = 0$ is shown in Figure~\ref{fig:acrobOnCart}. In order to compute the outer approximations we took a third order Taylor expansion of the non-polynomial dynamics even though exact treatment would be possible via a coordinate transformation leading to rational dynamics to which our methods can be readily extended; this extension is, however, not treated in this paper and therefore we opted for the simpler (and non-exact) approach using Taylor expansion. Before solving the problem we made a linear coordinate transform so that the state constraint set becomes the unit hypercube $[-1,1]^6$.

This example, which is the largest of those considered in this paper,  took 110 seconds to solve\footnote{All examples were run on an Apple iMac with 3.4 GHz Intel Core i7, 8 GB RAM and Mac OS X 10.8.2. The time reported is the pure solver time, not including the Yalmip preprocessing time.} with SeDuMi for $d = 4$; the corresponding time with the MOSEK SDP solver was 10 seconds. Using MOSEK we could also solve this example for $d = 6$ (in 420 seconds) although there the solver converged to a solution with a rather poor accuracy\footnote{Note that the MOSEK SDP solver is still being developed and its accuracy is likely to improve in the future.} and therefore we do not report the results.

\begin{figure*}[ht]
	\begin{picture}(120,120)
	\put(158,0){\includegraphics[width=60mm]{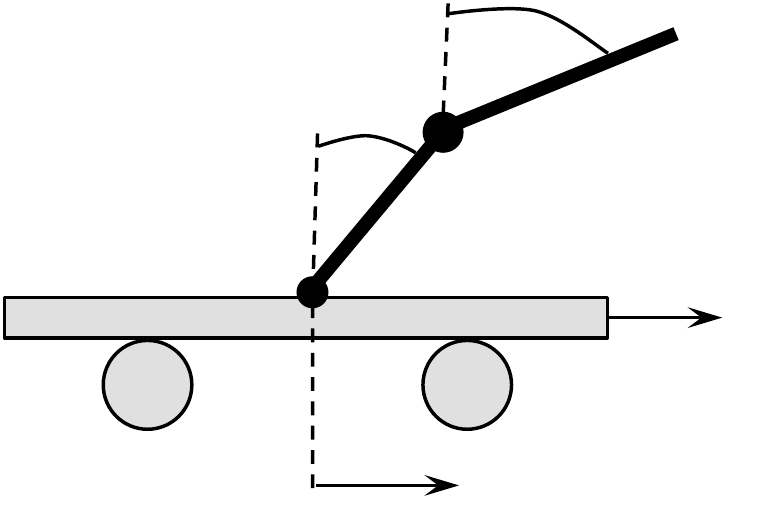}} 
	\put(238,15){\small $x_1$}
	\put(300,53){\small $u_1$}
	\put(260,78){\small $u_2$}
	\put(232,78){\small $x_2$}
	\put(268,107){\small $x_3$}
	\end{picture}
	\caption{\small Acrobot on a cart -- sketch}
	\label{fig:acrobOnCart_sketch}
\end{figure*}
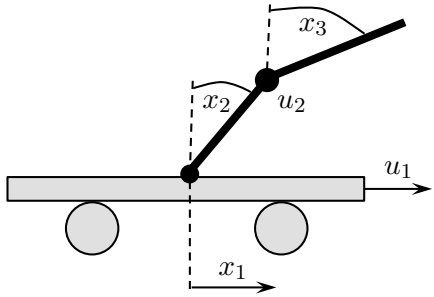

\begin{figure*}[ht]
	\begin{picture}(140,220)
	\put(128,20){\includegraphics[width=80mm]{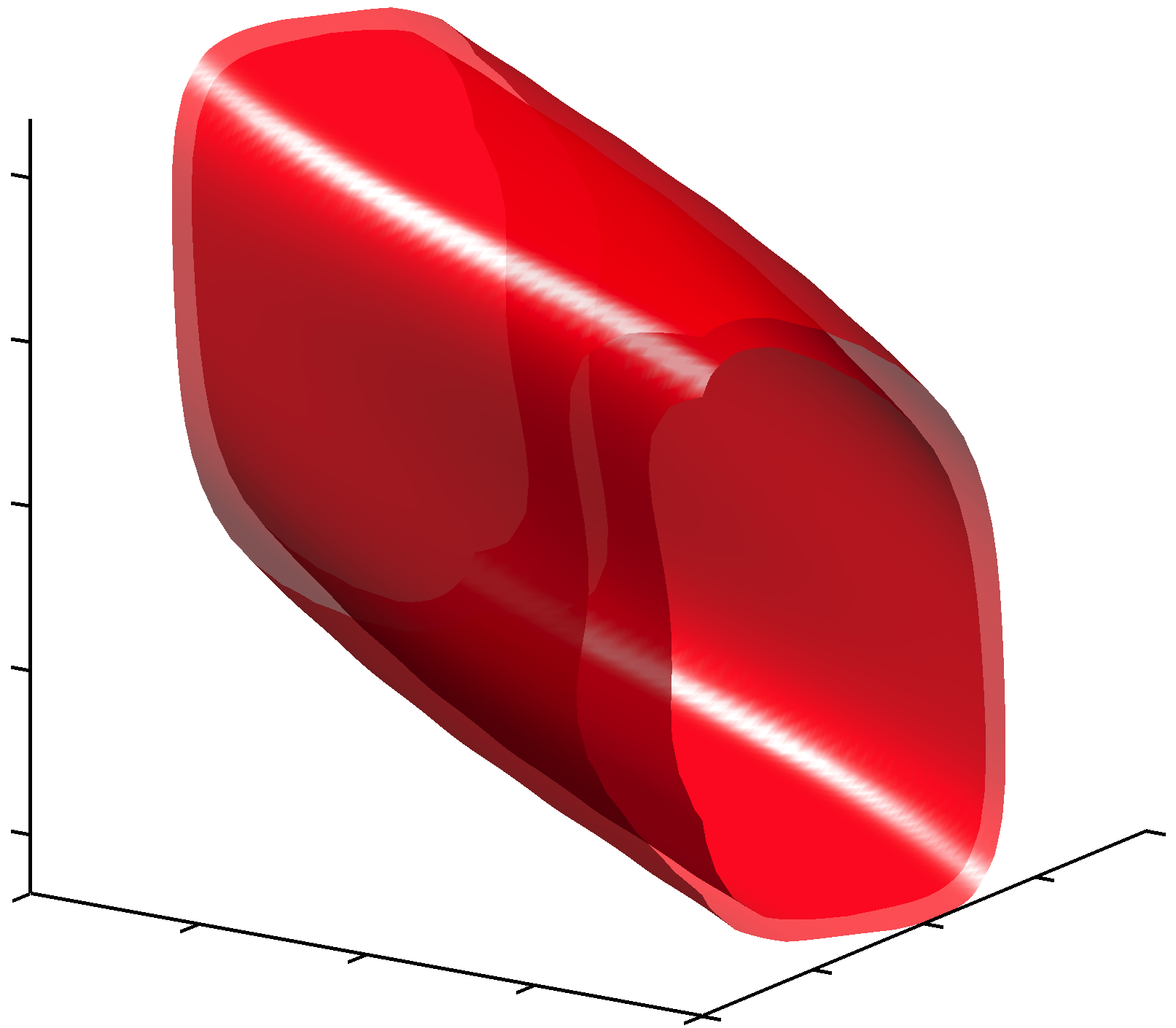}} 

	\put(165,10){\small $x_3$}
	\put(335,20){\small $x_2$}
	\put(70,119){\small $x_6 = \dot{x}_3$} 
	
	\put(110,35){\footnotesize $-\pi/2$}
	\put(140,28.5){\footnotesize $-\pi/4$}
	\put(184,22){\footnotesize $0$}
	\put(205,18){\footnotesize $\pi/4$}
	\put(242,12){\footnotesize $\pi/2$}
	
	\put(271,16){\footnotesize $-\pi/2$}
	\put(292.5,25){\footnotesize $-\pi/4$}
	\put(322.5,31.5){\footnotesize $0$}
	\put(332,41){\footnotesize $\pi/4$}
	\put(350,49){\footnotesize $\pi/2$}
	
	\put(114,55){\footnotesize $-1$}
	\put(107,87){\footnotesize $-0.5$}
	\put(120,118){\footnotesize $0$}
	\put(115,150){\footnotesize $0.5$}
	\put(121,182){\footnotesize $1$}

	\end{picture}
	\caption{\small Acrobot on a cart -- section of the polynomial outer approximations of degree four for $(x_1,x_4,x_5) = (0,0,0)$. Only the middle joint actuated -- darker, smaller; middle joint and the cart actuated -- lighter, larger. The states displayed $x_2$, $x_3$ and $x_6$ are, respectively, the lower pendulum angle, the upper pendulum angle and the upper pendulum angular velocity. }
	\label{fig:acrobOnCart}
\end{figure*}
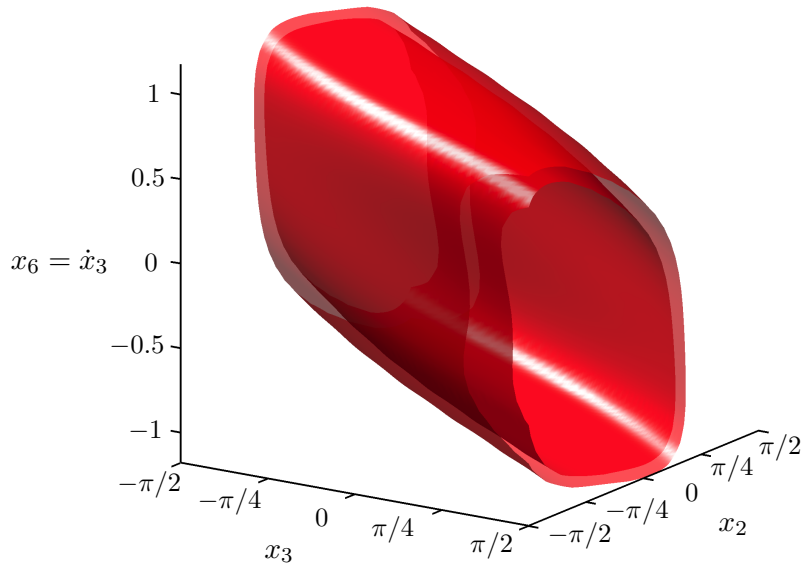

\section{Conclusion}

We derived an infinite-dimensional convex characterization of the maximum controlled invariant (MCI) set, finite-dimensional approximations of (the dual of) which provide a converging sequence of semialgebraic outer-approximations to this set. The outer-approximations are the outcome of a single semidefinite program
(SDP) with no additional data required besides the problem description. Therefore the approach is readily applicable using freely available modeling tools such Gloptipoly 3~\cite{glopti} or YALMIP~\cite{yalmip} with no hand-tuning involved.

The cost to pay for this comfort is the relatively unfavourable scalability of the semidefinite programs solved -- the number of variables grows as $O((n+m)^d)$, where $n$ and $m$ are the state and control dimensions and $d$ is the degree of the approximating polynomial. Therefore, in order for this approach to scale to medium dimensions (say, more than $m+n = 6$) one either has to tradeoff accuracy by taking small $d$ or go beyond the standard freely available solvers such as SeDuMi or SDPA. One possibility is parallelization; for instance, the free parallel solver SDPARA~\cite{sdpara} allows for the approach to scale to larger dimensions. Alternatively, one can utilize one of the (few) commercial SDP solvers; in particular, the recently released MOSEK SDP solver seems to show far superior performance on our problem class, and therefore this may allow the approach to scale to larger dimensions (see also the discussion following the acrobot-on-a-cart example in Section~\ref{sec:acrobotOnCart}). Finally, one can resort to customized structure-exploiting solutions; this is a promising direction of future research currently investigated by the authors. At this point it should be emphasized that, to the best of the authors' knowledge, all of the existing approaches providing approximations of similar quality experience similar or worse scalability properties.

Other directions of future research include the extension of the presented approach to inner approximations of MCI sets, to stochastic systems and to uncertain systems. Partial results on the inner approximations for the related problem of region of attraction computation already exist~\cite{roa_inner}, albeit in uncontrolled setting only.

\section*{Appendix A}
We start by embedding our problem in the setting of discrete-time Markov control processes; terminology and notation is borrowed from the classical reference~\cite{discrete}. Let us define a stochastic kernel
on $U$ given $X$ as a map $\nu(\cdot \!\mid\! \cdot)$ 
such that $\nu(\cdot\!\mid\! x)$ is a probability measure on $U$
for all $x\in X$ and $\nu(B\!\mid\! \cdot)$ is a measurable function on $X$ for all $B\subset U$.
Any such stochastic kernel gives rise to a discrete-time Markov process when applied to system~(\ref{sysd}) as a stationary randomized control policy (a policy which, given $x$, chooses the control action randomly based on the probability distribution $\nu(\cdot\!\mid\! x )$, i.e., $\mathrm{Prob}(u\in B\!\mid\! x)= \nu(B\!\mid\! x)$ for all $B\subset U$). The transition kernel $Q_\nu(\cdot\!\mid\! \cdot)$ of this stationary Markov process is then given by
\[
Q_\nu(A\!\mid\! x) = \int_U I_A(f(x,u))\,d\nu(u\!\mid\! x) =  \mathrm{Prob}(x^+\in A\!\mid\! x)\quad\forall\,A\subset\mathbb{R}^n,
\]
where $x$ is the current state and $x^+$ the successor state. The $t$-step transition kernel is then defined by induction as
\[
Q_{\nu}^t(A\!\mid\! x) := \int_{\mathbb{R}^n}Q(A \!\mid \!   y)\,dQ_{\nu}^{t-1}(y \!\mid \!   x),\quad t\in\{2,3,\ldots\}
\]
with $Q_\nu^1 := Q_\nu$. Given an initial distribution $\mu_0$, the distribution of the Markov chain at time $t$, $\tilde{\mu}_t$, is given by
\[
	\tilde{\mu}_t(A) = \int_X Q_\nu^t(A\!\mid\! x)\,d\mu_0(x) = \mathrm{Prob}(x_t\in A).
\]
The joint distribution of state and control is then
\[
\mu_t(A\times B) = \int_A \nu(B\!\mid\! x)\,d\tilde{\mu}_t(x).
\]
The discounted occupation measure associated to the Markov process is defined by
\[
\mu(A\times B) = \sum_{t=0}^\infty \alpha^t \mu_t(A\times B).
\]
Note that this relation reduces to~(\ref{eq:discOM}) when $\mu_t = \delta_{(x_t,u_t)}$

In order to prove Lemma~\ref{lem:corrDisc} we need the following result that can be found in~\cite{discrete}. 

\begin{lemma}\label{lem:corrd}
For any pair of measures $(\mu_0,\mu)$ satisfying equation~(\ref{eq:LiouvDisc}) there exists a stationary randomized control policy $\nu(\cdot\!\mid\!x)$ such the Markov chain obtained by applying this control policy to the difference equation~(\ref{sysd}) starting from initial distribution $\mu_0$ has the discounted occupation measure equal to $\mu$.
\end{lemma}
\begin{proof}
Disintegrate $\mu$ as $d\mu(x,u) = d\nu(u \!\mid \!   x)d\tilde{\mu}(x)$, where $\tilde{\mu}$ denotes the $x$-marginal of $\mu$ and $\nu$ is a stochastic kernel on $U$ given $X$. According to the discussion preceding Lemma~\ref{lem:corrd}, applying $\nu$ to (\ref{sysd}) gives rise to a stationary discrete-time Markov process with the transition kernel $Q_\nu$ starting form the initial distribution $\mu_0$.

With this notation, equation~(\ref{eq:LiouvDisc}) can be equivalently rewritten as
\begin{equation}\label{eq:LiouvDisc2}
\int_X v(x)\,d\tilde{\mu}(x) = \int_X v(x)\,d\mu_0(x) + \alpha \int_X\int_{\mathbb{R}^n} v(y)\,dQ_\nu(y \!\mid \!   x)\,d\tilde{\mu}(x)
\end{equation}
for any measurable $v(x)$ (derivation of equation~(\ref{eq:LiouvDisc}) did not depend on the continuity of~$v$). Taking $v(x) := I_A(x)$ we obtain 
\begin{equation}\label{eq:LiouvDisc3}
\tilde{\mu}(A) = \mu_0(A) + \alpha\int_X Q_\nu (A \!\mid \!   x)\,d\tilde{\mu}(x)\quad \forall\, A\subset X.
\end{equation}
 
Using relation~(\ref{eq:LiouvDisc2}) with $v(x) := Q_\nu(A \!\mid \!   x)$ to evaluate the integral w.r.t. $\tilde{\mu}$ on the right hand side of~(\ref{eq:LiouvDisc3}) we get
\[
\tilde{\mu}(A) = \mu_0(A) + \alpha\int_XQ_\nu(A \!\mid \!   x)\,d\mu_0(x)+ \alpha^2\int_X Q_\nu^2(A \!\mid \!   x)\,d\tilde{\mu}(x).
\]
By iterating this procedure we obtain
\begin{equation}\label{eq:convToZero}
\tilde{\mu}(A) = \mu_0(A)+\sum_{i=1}^t \alpha^i \underbrace{\int_X Q_\nu^i(A \!\mid \!   x)\,d\mu_0(x)}_{\displaystyle \mu_i(A)}\; +\; \underbrace{\alpha^{t+1}\int_X Q_\nu^{t+1}(A \!\mid \!   x)\,d\tilde{\mu}(x)}_{\displaystyle \to 0},
\end{equation}
and taking the limit as $t\to \infty$ gives
\[
\tilde{\mu}(A) = \sum_{t=0}^\infty \alpha^t\tilde{\mu}_t(A),
\]
where the third term in~(\ref{eq:convToZero}) converges to zero because $\alpha\in(0,1)$, $Q_\nu^{t+1}(A\!\mid\!x) \le 1$ and $\tilde{\mu}$ is a finite measure. Hence the $x$-marginal of the discounted occupation measure of the Markov chain coincides with the $x$-marginal of $\mu$.

Finally, to establish equality of the whole measures observe that
\[\sum_{t=0}^\infty \alpha^t\mu_t(A\times B) = \sum_{t=0}^\infty\alpha^t \int_A \nu(B \!\mid \!   x)\,d\tilde{\mu}_t(x)= \int_A \nu(B \!\mid \!   x)\,d\tilde{\mu}(x)=\mu(A\times B).  \]
\end{proof}

\textbf{Proof of Lemma~\ref{lem:corrDisc}:} Disintegrate $\mu$ to $d\mu(x,u) = d\nu(u \!\mid\! x)d\tilde{\mu}(x)$ as in the proof of Lemma~\ref{lem:corrd}. Then for any $x\in S:=\mathrm{spt}\,\tilde{\mu}$ we have
\[\int_U I_{S}(f(x,u))\,\nu(u \!\mid \!   x) = 1.\]
This relation says that the support of $\tilde{\mu}$ is invariant under $\nu$ and follows from Lemma~\ref{lem:corrd}, from the definition of the occupation measure $\mu$, from the definition of the support and from the fact that $\nu(\cdot \!\mid \!   x)$ is a probability measure for all $x$.

Define an admissible stationary deterministic control policy by taking any measurable selection $u(x)\in \mathrm{spt}\,\nu(\cdot  \!\mid \!   x) \subset U$. Define further the sequence of probability measures
\[
\nu_n(A \!\mid \!   x) = \frac{\nu(B_{1/n}(u(x))\cap A \!\mid \!   x)}{ \nu(B_{1/n}(u(x))\cap U \!\mid \!   x)}\quad \forall\, n\in\{1,2,\ldots\},\;\; A\subset U,
\]
where $B_{1/n}(u(x))$ is a closed ball of radius $1/n$ centered at $u(x)$. Then $\nu_n(\cdot\!\mid\! x)$ converges weakly-* (or weakly or narrowly) to $\delta_{u(x)}$ and
\[
\int_U I_{S}(f(x,u))\,\nu_n(u \!\mid \!   x) = 1 \quad \forall\,n\in\{1,2,\ldots\}. 
\]
Therefore,
\[
1  = \limsup_{n\to\infty} \int_U I_{S}(f(x,u))\,\nu_n(u \!\mid \!   x) \le\int_U I_{S}(f(x,u))\,\delta_{u(x)}(u) = I_{S}(f(x,u(x))),
\]
where the inequality follows by the Portmanteau lemma since the set $\{u \!\mid \!   f(x,u)\in S\cap B_{1/n}(u(x))\}$ is closed for all $x$ by continuity of $f$. Therefore in fact $I_{S}(f(x,u(x)))=1$ and so $f(x,u(x))\in \mathrm{spt}\,\tilde{\mu}$ for all $x\in \mathrm{spt}\, \tilde{\mu}$. Therefore $\mathrm{spt}\,\tilde{\mu}\subset X$ is invariant for the closed loop system $x_{t+1} = f(x_t,u(x_t))$, where $u(x)$ is an admissible deterministic control policy. Therefore necessarily $\mathrm{spt}\,\tilde{\mu}\subset X_I$. Finally, from equation~(\ref{eq:LiouvDisc}) clearly $\mathrm{spt}\,\mu_0\subset \mathrm{spt}\,\tilde{\mu}$ and so $\mathrm{spt}\,\mu_0\subset X_I$.  \hfill $\Box$

\section{Appendix~B}
\begin{lemma}\label{lem:corrContAux}
For any pair of measures ($\mu_0,\mu$) solving (\ref{eq:LiouvCont}), there exists a family of trajectories of the convexified inclusion~(\ref{sysc}) starting from $\mu_0$ such that the $x$-marginal of its discounted occupation measure is equal to the $x$-marginal of $\mu$.
\end{lemma}
\begin{proof}
The proof is based on fundamental results of~\cite{ambrosio} and \cite{bhatt} and on the compactification procedure discussed in~\cite{kurtz}.

We begin by embedding the problem in a stochastic setting. To this end, define the extended state space $E$ as the one-point compactification of $\mathbb{R}^n$, i.e., $E = \mathbb{R}^n \cup \{\Delta\}$, where $\Delta$ is the point compactifying $\mathbb{R}^n$. Define also the linear operator $A:\mathcal{D}(A) \to C(E\times U)$ by
\[w\mapsto A w := \mathrm{grad}\,{w}\cdot f,\]
where the domain of $A$, $\mathcal{D}(A)$, is defined as
\begin{align*}
\mathcal{D}(A):= \{ w : E\to \mathbb{R}  \mid  &\; w \in C^1(\mathbb{R}^n),\; w(\Delta) = 0, \lim_{x \to\Delta}w(x) = 0,\\ & \lim_{x\to\Delta}\mathrm{grad}\,w\cdot f(x,u) = 0\; \forall\: u\in U\}.
\end{align*}
In words, $\mathcal{D}(A)$ is the space all continuously differentiable functions vanishing at infinity such that $\mathrm{grad}\,w\cdot f$ also vanishes at infinity for all $u\in U$. Now consider the relaxed martingale problem~\cite{bhatt}: find a stochastic process $Y : [0,\infty]\times \Omega\to E$ defined on some filtered probability space $(\Omega,\mathcal{F}, (\mathcal{F}_t)_{t\ge 0}, P)$ and a stochastic kernel $\nu(\cdot\!\mid\!\cdot)$ (stationary relaxed Markov control) on $U$ given $E$ such that
\begin{itemize}
\item $P(Y(0) \in A) = \mu_0(A)\quad \forall\, A\subset E$
\item for all $w\in \mathcal{D}(A)$ the stochastic process
\begin{equation}\label{eq:mart}
w(Y(t)) - \int_0^t \int_U Aw(Y(\tau),u)\,\nu(du  \!\mid\!  Y(\tau))\,d\tau
\end{equation}
is an $\mathcal{F}_t$-martingale (see, e.g., \cite{kallenberg} for a definition).
\end{itemize}

Observe that there exists a countable subset of $\mathcal{D}(A)$ (e.g., all polynomials with rational coefficients attenuated near infinity) dense in $\mathcal{D}(A)$ in the supremum norm. Next, $\mathcal{D}(A)$ is clearly an algebra that separates points of $E$ and $A1 = 0$. Finally, since $f(x,u)$ is polynomial and hence locally Lipschitz, the ODE $\dot{x} = f(x,u)$ has a solution on $[0,\infty)$ for any $x_0\in E$ and any fixed $u\in U$ in the sense that if there is a finite escape time $t_e$, then we define $x(t) = \Delta$ for all $t\ge t_e$. Each such solution satisfies the martingale relation~(\ref{eq:mart}) (with a trivial probability space). Therefore,  $A$ satisfies Conditions 1-3 of \cite{bhatt} and it follows from Theorem~2.2 and Corollary 2.2 therein that for any pair of measures satisfying the discounted Liouville's equation~(\ref{eq:LiouvCont}), there exists a solution to the above martingale problem whose discounted occupation measure is equal to $\mu$, that is,
\[
\mu(A\times B) =  \E\Big\{  \int_{0}^\infty e^{-\beta t}I_{A\times B}(Y(t),u)\,\nu(du \!\mid\!  Y(t))\,dt  \Big\},\quad P(Y(0)\in A) = \mu_0(A),
\]
where $\E$ denotes the expectation w.r.t. the probability measure $P$. From the martingale property of~(\ref{eq:mart}) and the definition of $A$ we get
\[
\E\{w(Y(t))\} - \E\Big\{ \int_0^t \int_U \mathrm{grad}\,w\cdot f(Y(\tau),u)\,\nu(du \!\mid\!  Y(\tau))\,d\tau \Big\} = \E\{Y(0)\}.
\]
Now let $\mu_t$  denote the marginal distribution of $Y(t)$ at time $t$; that is, \[\mu_t(A) := P(Y(t)\in A) = \E\{I_A(Y(t))\} \quad \forall\: A\subset X.\] Then the above relation becomes
\[
\int_X w(x)\, d\mu_t(x) - \int_0^t \int_X \int_U \mathrm{grad}\,w(x)\cdot f(x,u)\,\nu(du \!\mid\!  x)\,d\mu_\tau(x)\,d\tau = \int w(x)\,d\mu_0(x),
\]
where we have used Fubini's thorem to interchange the expectation operator and integration w.r.t. time. Defining the relaxed vector field
\[
\bar{f}(x) = \int_U f(x,u)\,\nu(du \!\mid\!  x)\in \mathrm{conv}\,f(x,U)
\]
and rearranging we obtain
\begin{equation}\label{eq:contInteg}
\int_X w(x)\, d\mu_t(x)  = \int w(x)\,d\mu_0(x) + \int_0^t \int_X \mathrm{grad}\,w(x)\cdot \bar{f}(x)\,d\mu_\tau(x)\,d\tau,
\end{equation}
where the equation holds for all $w\in C^1(X)$ almost everywhere with respect to the Lebesgue measure on $[0,\infty)$. The Lemma then follows from Ambrosio's superposition principle~\cite[Theorem~3.2]{ambrosio} using the same arguments as in the proof of Lemma~4 in~\cite{roa}.

\end{proof}

\textbf{Proof of Lemma~\ref{lem:corrCont}:} Suppose that a pair of measures $(\mu_0,\mu)$ satisfies~(\ref{eq:LiouvCont}) and that $\lambda(\mathrm{spt}\,\mu_0 \setminus X_I) > 0$. From Lemma~\ref{lem:corrContAux} there is a family of trajectories of~(\ref{sysc}) starting from $\mu_0$ with discounted occupation measure whose $x$-marginal coincides with the $x$-marginal of $\mu$. However, this is a contradiction since no trajectory starting from $\mathrm{spt}\,\mu_0 \setminus X_I$ remains in $X$ for all times and $\mathrm{spt}\,\mu\subset X$. Thus,  $\lambda(\mathrm{spt}\,\mu_0 \setminus X_I) = 0$ and so $\lambda(\mathrm{spt}\,\mu_0) \le \lambda(X_I)$. \hfill $\Box$

\section{Acknowledgements}
The authors are grateful to Sl\'avka Jadlovsk\'a for providing the acrobot-on-a-cart system and Andrea Alessandretti for providing the spider-web system.


\begin{thebibliography}{XX}

\bibitem{ahmadi_msc}
A. A. Ahmadi. Non-monotonic Lyapunov functions for stability of nonlinear and switched
systems: theory and computation. Master's Thesis, MIT, Boston, 2008.

\bibitem{ambrosio}
L. Ambrosio. Transport equation and Cauchy problem for non-smooth vector fields.
In L. Ambrosio et al. (eds.), Calculus of variations and nonlinear partial differential equations.
Lecture Notes in Mathematics, Vol. 1927, Springer-Verlag, Berlin, 2008.

\bibitem{anderson}
E. J. Anderson, P. Nash. Linear programming in infinite-dimensional spaces:
theory and applications. Wiley, New York, 1987.

\bibitem{ash}
R. B. Ash. Real analysis and probability. Academic Press, San Diego, CA, 1972.



\bibitem{aubin}
J. P. Aubin, H. Frankowska. Set-valued analysis. Springer-Verlag, Berlin, 1990.


\bibitem{girard}
M. A. Ben~Sassi, A. Girard.
Controller synthesis for robust invariance of polynomial dynamical systems
using linear programming. System Control Letters 61(4):506-512, 2012.

\bibitem{bertsekas72}
D. Bertsekas. Infinite time reachability of state-space regions by using feedback control. IEEE Trans. Autom. Control 17(5):604-613, 1972.

\bibitem{bhatt}
A. G. Bhatt, V. S. Borkar. Occupation Measures for Controlled Markov Process: Characterization and Optimality. Annals of Probability, 24:1531-1562, 1996.

\bibitem{blanchini}
F. Blanchini. Set invariance in control. Automatica, 35(11):1747-1767, 1999.

\bibitem{blanchini94}
F. Blanchini. Ultimate boundedness control for uncertain discrete time systems via set-induced Lyapunov functions. IEEE Trans. Autom. Control 39(2):428-433, 1994.

\bibitem{blanchini2008}
F. Blanchini, S. Miani, C. Savorgnan. Dynamic augmentation and
complexity reduction of set-based constrained control. Proc. IFAC World
Congress on Automatic Control, Seoul, South Korea, 2008.

\bibitem{blanco}
T. B. Blanco, M. Cannon, B. De Moor. On efficient computation of low-complexity controlled invariant sets for uncertain linear systems. Int. J. Control 83(7):1339-1346, 2010.

\bibitem{miani}
F. Blanchini, S. Miani. Set-theoretic methods in control. Birkh\"auser,
Boston, 2007.

\bibitem{dorea}
C. E. T. Dorea, J. C. Hennet. $(A,\,B)$-invariant polyhedral sets of linear
discrete-time systems. J. Optim. Theory Appl., 103(3):521-542, 1999.



\bibitem{chesi}
G. Chesi. Domain of attraction; analysis and control via SOS programming.
Lecture Notes in Control and Information Sciences, Vol. 415, Springer-Verlag, Berlin, 2011.



\bibitem{quincampoix}
V. Gaitsgory, M. Quincampoix.
Linear programming approach to deterministic infinite horizon optimal control problems with discounting.
SIAM J. on Control and Optimization, 48:2480-2512, 2009.




\bibitem{gilbert_tan}
E. G. Gilbert, K. T. Tan. Linear systems with state and control constraints:
the theory and application of maximal output admissible sets.
IEEE Trans. Autom. Control 36(9):1008-1020, 1991.


\bibitem{gond_aut09}
R. Gondhalekar, J. Imura, K. Kashima. Controlled invariant
feasibility -- A general approach to enforcing strong feasibility in {MPC}
applied to move-blocking. Automatica 45(12):2869-2875, 2009.

\bibitem{gutman_cwikel}
P. O. Gutman, M. Cwikel. An algorithm to find maximal state constraint sets for discrete-time linear dynamical systems with bounded controls and states. IEEE Trans. Autom. Control 32(3):251-254, 1987.

\bibitem{roa}
D. Henrion, M. Korda.
Convex computation of the region of attraction of polynomial control systems.
{\tt arXiv:1208.1751}, August 2012.


\bibitem{glopti}
D. Henrion, J. B. Lasserre, J. L\"ofberg. Gloptipoly 3: moments, optimization and
semidefinite programming. Optim. Methods and Software 24:761--779, 2009.

\bibitem{volume}
D. Henrion, J. B. Lasserre, C. Savorgnan. Approximate volume and integration
for basic semialgebraic sets. SIAM Review 51:722-743, 2009.

\bibitem{discrete}
O. Hern\'andez-Lerma, J. B. Lasserre. Discrete-time Markov control processes:
basic optimality criteria. Springer-Verlag, Berlin, 1996.



\bibitem{slavka}
S. Jadlovsk\'a, A. Jadlovsk\'a. Inverted pendula simulation and modeling -- a generalized approach. International Conference on Process Control, Kouty nad Desnou, Czech Republic, 2010.

\bibitem{kallenberg}
O. Kallenberg. Foundations of modern probability. Springer-Verlag, Berlin, 2010.

\bibitem{kanatnikov}
A. N. Kanatnikov, A. P. Krishchenko. Localization of compact invariant sets of discrete-time nonlinear systems. Int. J. Bifurcation and Chaos 21(7):2057-2065, 2011.

\bibitem{kerriganPHD}
E. C. Kerrigan. Robust constraint satisfaction: Invariant sets and predictive
control. Ph.D. Thesis. Univ. Cambridge, UK, 2000.

\bibitem{roa_inner}
M. Korda, D. Henrion, C. N. Jones.
Inner approximations of the region of attraction for polynomial dynamical systems.
{\tt arxiv.org/pdf/1210.3184}, October~2012.

\bibitem{chatala}
A. P. Krishchenko. A. N. Kanatnikov. Maximal compact positively invariant sets of discrete-time nonlinear systems. Proc. IFAC World Congress on Automatic Control, Milano, Italy, 2011.

\bibitem{kurtz}
T. G. Kurtz. Equivalence of stochastic equations and martingale problems.
Stochastic Analysis 2010, 113-130, Springer-Verlag, Berlin, 2011.


\bibitem{kurzhanski_2013}
A. N. Daryin, A. B. Kurzhanski. Parallel algorithm for calculating the invariant sets of high-dimensional linear systems under uncertainty. Computational Mathematics and Mathematical Physics 53(1):34-43, 2013.

\bibitem{lasserre}
J. B. Lasserre. Moments, positive polynomials and their applications.
Imperial College Press, London, UK, 2009.

\bibitem{henon}
M. Liu, S. Zhang, Z. Fan, M. Qiu. $\mathrm{H}_\infty$ State Estimation for Discrete-Time Chaotic Systems Based on a Unified Model. IEEE Trans. on Systems, Man, and Cybernetics -- Part B, Cybernetics, 42(4):1053-1063, 2012.

\bibitem{yalmip}
J. L\"ofberg. YALMIP : A toolbox for modeling and optimization in MATLAB. In Proc. IEEE CCA/ISIC/CACSD
Conference, Taipei, Taiwan, 2004. 



\bibitem{anirudha} A. Majumdar, A. A. Ahmadi, R. Tedrake. Control Design Along Trajectories with Sums of Squares Programming. IEEE International Conference on Robotics and Automation (ICRA), 2013 (to appear).

\bibitem{lygeros}
K. Margellos, J. Lygeros. Hamilton-Jacobi formulation for reach-avoid differential games.
IEEE Transactions on Automatic Control, 56:1849-1861, 2011.


\bibitem{mitchell}
I. Mitchell, C. Tomlin.
Overapproximating reachable sets by Hamilton-Jacobi projections.
Journal of Scientific Computing, 19:323-346, 2003.

\bibitem{sedumi}
I. P\'olik, T. Terlaky, Y. Zinchenko. SeDuMi: a package for conic optimization. IMA
workshop on Optimization and Control, Univ. Minnesota, Minneapolis, 2007. 

\bibitem{putinar}
M. Putinar. Positive polynomials on compact semi-algebraic sets. Indiana Univ. Mathematics Journal, 42:969-984, 1993.

\bibitem{rajaram}
R. Rajarama, U. Vaidya, M. Fardadc, B. Ganapathysubramanian. Stability in the almost everywhere sense: A linear transfer operator approach. Journal of Mathematical Analysis and Applications, 368:144-156, 2010.

\bibitem{rakovic}
S. V. Rakovi\'c. Parameterized robust control invariant sets for linear systems: theoretical advances
and computational remarks. IEEE Trans. Autom. Control, 55(7):1599-1614, 2010.

\bibitem{rantzer}
A. Rantzer. A dual to Lyapunov's stability theorem. Systems \& Control Letters, 42:161-168, 2001.

\bibitem{rubio}
J. E. Rubio. Control and Optimization: The Linear Treatment of Nonlinear Problems. Manchester University Press, Manchester, UK, 1985.

\bibitem{vidal}
R. Vidal, S. Schaert, J. Lygeros, S. Sastry. Controlled invariance of discrete time systems.
HSCC, Lecture Notes on Computer Science, 1790, Springer-Verlag, Berlin, 2000.

\bibitem{starkov1}
K. Starkov. Bounds for compact invariant sets of the system describing dynamics of the nuclear spin generator. Communications in Nonlinear Science and Numerical Simulation 14(6):2565-2570, 2009.

\bibitem{starkov2}
K. Starkov. Estimation of the domain containing all compact invariant sets of the optically injected laser system. Int. J. Bifurcation and Chaos 17(11):4213-4217, 2007.

\bibitem{sdpa}
M. Yamashita, K. Fujisawa, M. Fukuda, K. Kobayashi, K. Nakta, M. Nakata. Latest developments in the SDPA Family for solving large-scale SDPs. In M. Anjos, J. B. Lasserre  (Eds.). Handbook on Semidefinite, Cone and Polynomial Optimization: Theory, Algorithms, Software and Applications. Springer, NY, USA, Chap. 24, 687-714, 2011.

\bibitem{sdpara}
M. Yamashita, K. Fujisawa, M. Kojima. SDPARA: SemiDefinite Programming Algorithm paRAllel version. Parallel Computing 29:1053-1067, 2003.

\bibitem{young}
L. C. Young. Calculus of variations and optimal control theory. Sunders, Philadelphia, 1969.



\bibitem{topcu} U. Topcu, A. K. Packard, P. Seiler, G. J. Balas. Robust region-of-attraction estimation. IEEE Transactions on Automatic Control, 55:137-142, 2010.



\bibitem{vaidya}
K. Wang, U. Vaidya.
Transfer operator approach for computing domain of attraction. IEEE Conference on Decision and Control
(CDC), Atlanta, GA, 2010.



\end{thebibliography}
\end{document}